\documentclass{amsart}
\usepackage[T1]{fontenc}
\usepackage{verbatim}
\usepackage{amssymb,amscd,amsmath,amsthm}
\usepackage{mathtools}
\usepackage[margin=1.2in]{geometry}
\usepackage{indentfirst}
\usepackage{enumerate}
\usepackage{cases}

\numberwithin{equation}{section}
\newtheorem{thm}{Theorem}[section]
\newtheorem{lem}[thm]{Lemma}
\newtheorem{prop}[thm]{Proposition}
\newtheorem{cor}[thm]{Corollary}

\newtheorem{rmk}{Remark}[section]

\newcommand{\Z}{\mathbb{Z}}

\title{Squarefree numbers in arithmetic progressions}
\author{Ramon M. Nunes}
\address{Universit\'e Paris Sud, Laboratoire de math\'ematiques\\
Campus d'Orsay\\ 91405 Orsay Cedex\\ France}
\email{ramon.moreira@math.u-psud.fr}

\date{\today}

\begin{document}

\begin{abstract}
We give asymptotics for correlation sums linked with the distribution of squarefree numbers in arithmetic progressions over a fixed modulus. As a particular case we improve a result of Blomer \cite{Blomer} concerning the variance.
\end{abstract}
\maketitle

\section{Introduction}

For a positive real number $X$ and positive integers $a,q$ with $(a,q)=1$, let $E(X,q,a)$ be defined by the formula

\begin{equation}\label{E}
\sum_{\substack{n\leq X\\ n\equiv a\!\!\!\pmod q}}\mu^2(n)=\dfrac{6}{\pi^2}\prod_{p\mid q}\left(1-\frac{1}{p^2}\right)^{-1}\frac{X}{q} + E(X,q,a),
\end{equation}
where, as usual, $\mu$ is the M\"obius function.
In \eqref{E} the first term heuristically appears to be a good approximation of the number of squarefree integers $\leq X$ congruent to $a\pmod{q}$. In this paper, we are concerned with the so called error term $E(X,q,a)$. Trivially, one has
\begin{equation}\label{simple}
\left\vert E(X,q,a)\right\vert\leq \dfrac{X}{q}+1,
\end{equation}
while in \cite{Hooley}, Hooley proved

\begin{equation}\label{Hoo}
E(X,q,a)=O_{\epsilon}\left(\left(\dfrac{X}{q}\right)^{1/2}+q^{1/2+\epsilon}\right),
\end{equation}
where the $O_{\epsilon}$-constant depends only on $\epsilon>0$ arbitrary. This is the best result available for fixed $a$. Futhermore, formula \eqref{Hoo} gives an asymptotic formula for the left-hand side of \eqref{E} for $q\leq X^{\frac{2}{3}-\epsilon}$(We believe that such an asymptotic formula should hold for $q\leq X^{1-\epsilon}$ and it is a challenging problem to go beyond $X^{\frac{2}{3}-\epsilon}$ for a general $q$, in particular when $q$ is prime). The situation is quite similar to the equivalent problem for the divisor function $d(n)$, which is, roughly, justified by the appearance of exponential sums on it.\\
While for the divisor function we consider the classical Kloosterman sums

$$K(a,b;q):=\sum_{\substack{1\leq x\leq q-1\\ (x,q)=1}}e\left(\dfrac{ax+b\overline{x}}{q}\right),$$
for the $\mu^2$ function, we have to deal with

$$K_2(a,b;q):=\sum_{\substack{1\leq x\leq q-1\\ (x,q)=1}}e\left(\dfrac{ax+b\overline{x}^2}{q}\right),$$
where $e(x)=e^{2\pi ix}$ and $\overline{x}$ is the multiplicative inverse of $x\!\!\pmod{q}$.\\
Blomer \cite{Blomer} considered a certain average over the residue classes $\!\!\!\pmod q$. More precisely, he considered the following second moment(variance) of the $E(X,q,a)$

\begin{equation}\label{V}
\mathcal{M}_2(X,q)=\sum_{a\!\!\!\pmod q}\!\!\!\!\!{}^{*}\;\left|E(X,q,a)\right|^2,
\end{equation}
where the $*$ symbol means we only sum over the classes that are relatively prime to $q$.\\
In \cite[Theorem 1.3]{Blomer}, he showed that

\begin{equation}\label{BL}
\mathcal{M}_2(X,q)\ll X^{\epsilon}\left(X+\min\left(\dfrac{X^{5/3}}{q},q^2\right)\right)
\end{equation}
holds for every $\epsilon>0$, uniformly for $1\leq q\leq X$.\\
Several years before, Croft \cite{Croft} considered a variation of $\mathcal{M}_2(X,q)$ by summing not only over the classes relatively prime to $q$ but over all the classes $\!\!\!\pmod q$. Let

\begin{equation*}
\mathcal{M}_2^{'}(X,q)=\sum_{a\!\!\!\pmod q}\left\{\sum_{\substack{n\leq X\\n\equiv a\!\!\!\pmod q}}\mu^2(n)-\dfrac{\mu^2(d)q_0}{\varphi(q_0)}\dfrac{6}{\pi^2}\prod_{p\mid q}\left(1+p^{-1}\right)^{-1}\dfrac{X}{q}\right\}^2,
\end{equation*}
where $d=(a,q)$ and $q_0=q/d$. The last term between the curly brackets on the expression above can be seen as the expected value of the sum

$$
\sum_{\substack{n\leq X\\n\equiv a\!\!\!\pmod q}}\mu^2(n)
$$
when $a$ is not necessarily relatively prime to $q$. Observe that it reduces to the first term on the right-hand side of \eqref{E} whenever $(a,q)=1$.

By doing an extra average over $q$, Croft \cite[Theorem 2]{Croft} proved the following formula

\begin{equation}\label{croft-thm}
\sum_{q\leq Q}\mathcal{M}_2^{'}(X,q)=BX^{1/2}Q^{3/2}+O\left(X^{2/5}Q^{8/5}\log^{13/5} X+X^{3/2}\log^{7/2}X\right)\text{, }\big(Q\leq X\big)
\end{equation}
where $B$ is an explicit constant. For $Q\geq X^{2/3+\epsilon}$, \eqref{croft-thm} above gives us an asymptotic formula for
$$
\sum_{q\leq Q}\mathcal{M}_2^{'}(X,q)
$$
which looks like the classical Barban-Davenport-Halbertam for primes with a main term (see Montgomery \cite{Montgomery}).
This result was further improved by several authors. We mention the works of Warlimont \cite{Warl2}, Br\"udern \textit{et al}. \cite{BGPVW} and Vaughan \cite{Vaughan}. It is worth mentioning that the later two results deal with the more general case of $k$-free numbers.

Ignoring for the moment the difference between $\mathcal{M}_2^{'}(X,q)$ and $\mathcal{M}_2(X,q)$, formula \eqref{croft-thm} above can be interpreted as saying that, at least on average over $q\leq Q$, the value of $\mathcal{M}_2(X,q)$ is much smaller than the bound given by \eqref{BL}.
In this paper we investigate if such phenomenon can be observed without the need of the extra average over $q$.
This is indeed possible for $q$ large enough relatively to $X$. Our main theorem goes in that direction

\begin{thm}\label{main}
Let $\mathcal{M}_2(X,q)$ be defined as in \eqref{V} and $\epsilon>0$ arbitrary. Then, uniformly for $q\leq X$, we have

\begin{equation}\label{maineq}
\mathcal{M}_2(X,q)= C\prod_{p\mid q}\bigg(1+2p^{-1}\bigg)^{-1}X^{1/2}q^{1/2}+O\left(d(q)X^{1/3}q^{2/3}+X^{23/15}q^{-13/15}(\log X)^{15}\right),
\end{equation}
where
\begin{equation}\label{C-}
C=\dfrac{\zeta\left(\frac{3}{2}\right)}{\pi}\prod_{p}\bigg(\frac{p^3-3p+2}{p^3}\bigg)=0.167\ldots
\end{equation}
and the $O$-constant is absolute.
\end{thm}

As pointed out to the author by Professor Rudnick, the same constant was found by Hall \cite{Hall} on the main term of the second moment for the problem of squarefree numbers in short intervals.\\
Our technique is more general since we can also provide a glimpse on the correlation between $E(X,q,a)$ and $E(X,q,ma)$ for fixed $m$ with $(m,q)=1$. For $X>0$, $m,q$ integers such that $m\neq 0$, $(m,q)=1$, $q\geq 1$ we define
\begin{equation}\label{m2m}
\mathcal{M}_2[m](X,q)= \sum_{a\!\!\!\pmod q}\!\!\!\!\!{}^{*}\;E(X,q,a)E(X,q,ma).
\end{equation}
The Theorem \ref{main} above can be deduced as a consequence of the following

\begin{thm}\label{main[m]}
Let $m$ be a nonzero integer of arbitrary sign. Also let $\mathcal{M}_2[m](X,q)$ be defined as in \eqref{m2m}. Then, uniformly for $q\leq X$, $(m,q)=1$ we have

\begin{multline}\label{maineq[m]}
\mathcal{M}_2[m](X,q)= \frac{C}{2}\Gamma_{\operatorname{an}}(m)\Gamma_{\operatorname{ar}}(m)\prod_{p\mid q}\bigg(1+2p^{-1}\bigg)^{-1}X^{1/2}q^{1/2}\\
+O_{m}\left(d(q)X^{1/3}q^{2/3}+X^{23/15}q^{-13/15}(\log X)^{15}\right),
\end{multline}
where $C$ is as in \eqref{C-}, where the analytic factor $\Gamma_{\operatorname{an}}(m)$ is defined by
\begin{equation}\label{Gamma-an}
\Gamma_{\operatorname{an}}(m):=
\begin{cases} 
\dfrac{\sqrt{m}+1-\sqrt{m-1}}{m}\text{, if }m>0,\\
\;\\
\dfrac{\sqrt{1-m}-\sqrt{-m}-1}{-m}\text{, if }m<0,
\end{cases}
\end{equation}
and the arithmetic factor $\Gamma_{\operatorname{ar}}(m)$ is
\begin{equation}\label{Gamma-ar}
\Gamma_{\operatorname{ar}}(m) := \prod_{p^2\mid m}\left(-\frac{p^2-1}{p+2}\right)\prod_{\substack{p\mid m\\p^2\nmid m}}\left(1+\frac{p+p^{1/2}+1}{p^{3/2}+p^{1/2}+1}\right)^{-1};
\end{equation}
and the $O_{m}$-constant depends at most on $m$.
\end{thm}
We state the result for general $m$ but we content ourselves with a complete proof only for $m$ squarefree.\\
Notice that making the choice $m=1$, one has $\Gamma_{\operatorname{an}}(1)=2,\Gamma_{\operatorname{ar}}(1)=1$ and we retrieve Theorem \ref{main}.
If one seeks for uniformity in $m$, the same techniques give an error term of the same order for $m\leq (\log X)^A$ for some constant $A$.\\
Notice that for $X^{8/13+\epsilon}<q<X^{1-\epsilon}$, formula \eqref{maineq} improves the upper bound \eqref{BL}. We have further the following direct consequence of Theorem \ref{main}

\begin{cor}
Let $\mathcal{M}_2(X,q)$ be defined as in \eqref{V} and $\epsilon>0$ arbitrary. Then as $X\rightarrow \infty$, we have

\begin{equation}\label{equiv}
\mathcal{M}_2(X,q) \sim C\prod_{p\mid q}\bigg(1+2p^{-1}\bigg)^{-1}X^{1/2}q^{1/2}
\end{equation}
uniformly for $X^{31/41+\epsilon}\leq q\leq X^{1-\epsilon}$ and where $C$ is as in \eqref{C-}.

\end{cor}
\begin{rmk}
The asymptotic formula \eqref{equiv} gives an average order of magnitude of $(X/q)^{\frac{1}{4}+\epsilon}$ for the terms $E(X,q,a)$. This remark goes in the direction of the conjecture due to Montgomery (see \cite[top of the page 145]{Croft}), which we write under the form
$$E(X,q,a)=O_{\epsilon}\left((X/q)^{\frac{1}{4}+\epsilon}\right)\text{, }\epsilon>0\text{ arbitrary}$$
uniformly for $(a,q)=1$, $X^{\theta_1}<q<X^{\theta_2}$ where the values of the constants $\theta_1$ and $\theta_2$ satisfying $0<\theta_1<\theta_2<1$ have to be precised.\\
\end{rmk}
Note that the same phenomenom can be observed in the work of Croft, with an extra average over $q$. Such a conjecture may, of course, be interpreted in terms of the poles of the functions $\frac{L(\chi,s)}{L(\chi^2,2s)}$, where $\chi$ is a Dirichlet character modulo $q$.
The error term $O(X^{23/15+\epsilon}q^{-13/15})$ in \eqref{maineq} comes from the use of the square sieve (see \cite{Heathbrown}). The proof could be simplified by avoiding the use of this sieve, obtaining a worse error term. One can obtain $O(X^{5/3+\epsilon}q^{-1})$ in a rather elementary way. That would imply that equivalence \eqref{equiv} would only hold for $X^{7/9+\epsilon}\leq q\leq X^{1-\epsilon}$.

%
%
\subsection{Discussion about $\Gamma_{\operatorname{an}}(m)$ and $\Gamma_{\operatorname{ar}}(m)$}
${}$\\
Consider the case where $m$ is squarefree. Formula \eqref{maineq[m]} shows that the random variables
$$
{\bf X}:a\!\!\!\!\pmod q\mapsto \dfrac{E(X,q,a)}{(X/q)^{1/2}}\text{ and }{\bf X}_m:a\!\!\!\!\pmod q\mapsto \dfrac{E(X,q,ma)}{(X/q)^{1/2}}
$$
are not asymptotically independent as $X\rightarrow \infty$, $q$ satisfying $X^{31/41+\epsilon}\leq q\leq X^{1-\epsilon}$. This is an easy consequence of the fact that these random variables have asymptotic mean equal to zero (see lemma \ref{lem-bl} below). The fact that ${\bf X}$ and ${\bf X}_m$ are dependent (when $m>0$) can be guessed in a similar way as in \cite[remark 1.8]{FGKM}, by the trivial fact that for a squarefree $n$ such that
$$
n\equiv a\!\!\!\!\!\pmod q, 1\leq n \leq X/m, (n,m)=1,
$$
then $n'=mn$ satisfies
$$
n'\text{ squarefree, }n'\equiv ma\!\!\!\!\!\pmod q, 1\leq n' \leq X.
$$
Such an interpretation obviously fails when $m<0$ or $m$ is not squarefree, which may explain the signs of $\Gamma_{\operatorname{an}}(m)$ and $\Gamma_{\operatorname{ar}}(m)$.
We also remark that the random variable
$$
\frac{1}{\phi(q)}\sum_{m\!\!\!\!\pmod q}\!\!\!\!\!\!\!{}^*\:\:{\bf X}{\bf X}_m
$$
has asmptotic mean zero, again by lemma \ref{lem-bl}.\\

Finally we would like to point out some differences between Theorem \ref{main[m]} and \cite[Theorem 1.5]{FGKM} from which this study was inspired.\\
From \cite[Corollary 1.7.]{FGKM}, the correlation for the divisor function exists if and only if $m>0$, with a correlation coefficient that is always positive. While in our setting, the correlation always exists and the sign depends on the value of $m$, corresponding to the sign of $m$, for $m$ squarefree.\\

In the work of Fouvry {\it et al.}\cite{FGKM} instead of only considering a homothety $a\mapsto ma$, they consider general M\"obius transformations
$$
a\mapsto \gamma(a):=\dfrac{m_1a+m_2}{m_3a+m_4}.
$$

The case of a general linear map $\gamma(a)=m_1a+m_2$ will be treated by a very different method in a posterior paper by the author. The techniques unfortunately do not extend to the reamaining cases, even for the seemingly simple case when $\gamma(a)=1/a$.

\subsection{A double sum over squarefree integers}

Developping the squares in $\mathcal{M}_2[m](X,q)$, we obtain the equality

\begin{equation}\label{developV}
\mathcal{M}_2[m](X,q)=S[m](X,q) - 2C(q)\dfrac{X}{q}\sum_{\substack{n\leq X\\ (n,q)=1}}\mu^2(n)+\varphi(q)\left(C(q)\dfrac{X}{q}\right)^2,
\end{equation}
where
\begin{equation}\label{cq}
C(q)=\dfrac{6}{\pi^2}\displaystyle\prod_{p\mid q}\left(1-\dfrac{1}{p^2}\right)^{-1},
\end{equation}
and $S[m](X,q)$ is the double sum

\begin{equation}\label{S2}
S[m](X,q)=\underset{\substack{n_1,n_2\leq X\\ (n_1n_2,q)=1\\mn_1=n_2\!\!\!\pmod q}}{\sum\sum}\mu^2(n_1)\mu^2(n_2).
\end{equation}
As in the classical dispersion method, it is important to obtain an asymptotic formula for each of the terms in \eqref{developV}, for $S[m](X,q)$ in particular. We point out that since the order of magnitude of $\mathcal{M}_2[m](X,q)$ is much smaller than some of the terms in the sum \eqref{developV}, our asymptotic formula for $S[m](X,q)$ must be precise enough to produce huge cancellations between the main terms. This precision is contained in
 
\begin{thm}\label{intermediate}
Let $X>2$ be a real number and let $m$ be a nonzero integer of arbitrary sign. Then for every $q$ integer, satisfying $q\leq X$, we have

\begin{multline}\label{intereq}
S[m](X,q)=\dfrac{\varphi(q)}{2}\left(\dfrac{C(q)X}{q}\right)^2+\frac{C}{2}\Gamma_{\operatorname{an}}(m)\Gamma_{\operatorname{ar}}(m)\prod_{p\mid q}\bigg(1+2p^{-1}\bigg)^{-1}X^{1/2}q^{1/2}\\
+O_{m}\left(d(q)X^{1/3}q^{2/3}+X^{23/15}q^{-13/15}(\log X)^{15}\right),
\end{multline}
where $C$, $C(q)$ are as in \eqref{C-} and \eqref{cq} respectively, $\Gamma_{\operatorname{an}}(m)$ and $\Gamma_{\operatorname{ar}}(m)$ are as in \eqref{Gamma-an} and \eqref{Gamma-ar} respectively; and the implied $O_{\epsilon,m}$-constant depends at most on $\epsilon$ and $m$.
\end{thm}

Some parts of the proof of Theorem \ref{intermediate} were inspired by the paper of Heath-Brown \cite{Heathbrown} where he proved (see \cite[Theorem 2]{Heathbrown}) that for every $X>2$, one has

\begin{equation}\label{HB}
\displaystyle\sum_{n<X}\mu^2(n)\mu^2(n+1)=C_2X+O(X^{7/11}(\log X)^7),
\end{equation}
where $C_2$ is given by

\begin{equation}\label{C_2}
C_2=\prod_{p}\bigg(1-2p^{-2}\bigg)=0.322\:634\:098\ldots
\end{equation}
For an improve on formula \eqref{HB}, see Reuss \cite{Reuss}. The proof of \eqref{intereq} will cover the main part of this article and the Theorem \ref{main[m]} will follow.


\section*{Acknowledgements}
I am most thankful to my advisor \'Etienne Fouvry for his helpful comments on this work. I would also like to express my gratitude to Sary Drappeau and Djordjo Milovic for reading and making valuable remarks; and to Zeev Rudnick for a stimulating conversation about Theorem \ref{main}.

\section{Notation}
For $X>1$ we write $\mathcal{L}:=\log 2X$\\
Let $\omega$ be the well-known function defined by
\begin{equation}
\omega(n)=\displaystyle\sum_{p\mid n}1,
\end{equation}
for every integer $n\geq 1$.
For $N>0$, we use the notation $n\sim N$ meaning $N<n\leq 2N$.\\
Let $\psi(v)$ be the sawtooth function defined, as in Titchmarsh \cite{Titchmarsh}, by
\begin{equation}\label{psi}
\lfloor v\rfloor=v-\frac{1}{2}+\psi(v).
\end{equation}
For $n\in \Z, n\neq 0$ we define
\begin{equation}\label{sigma}
\sigma(n)=\prod_{p^2\mid n}p.
\end{equation}

\section{Initial Steps}

We are considering
\begin{equation*}
S[m](X,q)=\underset{\substack{n_1,n_2\leq X\\ (n_1n_2,q)=1\\mn_1=n_2\!\!\!\pmod q}}{\sum\sum}\mu^2(n_1)\mu^2(n_2).
\end{equation*}

We write $\ell q=n_2-mn_1$ and write $S[m](X,q)$ as a sum over $\ell$. We have

\begin{equation}\label{sum-l}
S[m](X,q)=\sum_{\ell\in \Z}\sum_{\substack{n\in I(\ell)\\(n,q)=1}}\mu^2(n)\mu^2(mn+\ell q),
\end{equation}
where $I(\ell)=I(\ell;m,q)$ is the interval defined by

\begin{equation}
I(\ell)=
\begin{cases}
(0,X)\cap\left(\frac{-\ell q }{m},\frac{X-\ell q}{m}\right)\text{ if }m>0\\  
(0,X)\cap\left(\frac{X-\ell q}{m},\frac{-\ell q}{m}\right)\text{ if }m<0.
\end{cases}
\end{equation}
That is,
$$
I(\ell)=\{n; n\text{ and }mn+\ell q\in (0,X)\}.
$$

\begin{rmk}
Observe that since $\mu^2(n)\mu^2(mn+h)=1$ if and only if $\sigma(n)\sigma(mn+h)=1$ (recall definition \eqref{sigma} of $\sigma(n)$), we have
\begin{equation}\label{divisors-xi}
\mu^2(n)\mu^2(mn+h)=\sum_{d\mid \sigma(n)\sigma(mn+h)}\mu(d).
\end{equation}

\end{rmk}
We now replace formula \eqref{divisors-xi} with $h=\ell q$ in equation \eqref{sum-l} and invert order of summation. We thus obtain

\begin{equation}\label{S=muNd}
S[m](X,q)=\sum_{\ell \in \Z}\sum_{\substack{1\leq d\leq X\\(d,q)=1}}\mu(d)N_d(\ell),
\end{equation}
where
\begin{equation}\label{Nd}
N_d(\ell):=\#\{n\in I(\ell); (n,q)=1\text{ and } \sigma(n)\sigma(mn+\ell q)\equiv 0\!\!\!\pmod d\}.
\end{equation}

Let $0<y<X$ be a parameter to be chosen later depending on $X$ and $q$. We break down the sum \eqref{S=muNd} as follows
\begin{equation}\label{S'+S''}
 S[m](X,q)=S_{\leq y}[m](X,q)+S_{>y}[m](X,q),
\end{equation}
where,
\begin{equation}\label{S1,S2=}
\begin{cases}
 S_{\leq y}[m](X,q)=\displaystyle\sum_{\ell \in \Z}\displaystyle\sum_{\substack{d\leq y\\(d,q)=1}}\mu(d)N_d(\ell),\\
 S_{>y}[m](X,q)=\displaystyle\sum_{\ell \in \Z}\displaystyle\sum_{\substack{y< d\leq X\\(d,q)=1}}\mu(d)N_d(\ell).
\end{cases}
\end{equation}


From this point on, we make the hypothesis
$$
m\text{ is squarefree.}
$$
For the general proof one would need to distinguish two different cases in the definition of the function $\kappa$ (see \eqref{kappa} below). In the next lemma we study the first of the sums on the right-hand side of \eqref{S'+S''}.

\begin{lem}\label{lem<y}
Let $m\neq 0$ be a squarefree number and $y>1$. Let $S_{\leq y}[m](X,q)$ be defined by \eqref{S1,S2=}. Then for every $y>1$, we have
\begin{equation}\label{eq-lem<y}
S_{\leq y}[m](X,q)=\sum_{\ell\in \Z}f_q(\ell,m)\left|I(\ell)\right|+O_m\left(\dfrac{X}{q}\left(\tau(q)y+ Xy^{-1}\right)\log y\right),
\end{equation}
uniformly for $X>1$ and $q\geq 1$ satisfying $(m,q)=1$, where
\begin{equation}\label{fq}
f_q(\ell,m)=\prod_{p}\left(1-\frac{2}{p^2}\right)\prod_{p\mid m}\left(\frac{p^2-1}{p^2-2}\right)\prod_{p\mid q}\left(\frac{p^2-p}{p^2-2}\right)\kappa((\ell,m^2))\prod_{\substack{p^2\mid \ell\\p\nmid mq}}\left(\frac{p^2-1}{p^2-2}\right),
\end{equation}
$\kappa$ is the multiplicative function such that
\begin{align}\label{kappa}
\kappa(p^{\alpha})=
\begin{cases}
 \dfrac{p^2-p-1}{p^2-1},& \text{ if } \alpha=1,\\
 \:\:\:\:\dfrac{p^2-p}{p^2-1},& \text{ if } \alpha=2,\\
 \:\:\:\:\:\:\:\:\:\:0,& \text{ if } \alpha\geq 3,
\end{cases}
\end{align}
and the implied constant depends at most on $m$.
\end{lem}

\begin{proof}

We want to evaluate $S_{\leq y}[m](X,q)$. Let $p\nmid q$, we define
\begin{equation}
 u_p(\ell)=\#\{v\!\!\!\pmod{p^2}; v\equiv0\!\!\!\pmod{p^2}\text{ or }mv+\ell q\equiv0\!\!\!\pmod{p^2}\}.
\end{equation}
Again, we omit the dependence in $m$ and $q$ since both are held fixed and no confusion should arise from this.
By the Chinese Remainder Theorem, we have, for every $d$ squarefree, $(d,q)=1$.
\begin{equation}\label{Nd=}
N_d(\ell)=\frac{\varphi(q)}{q}U_d(\ell)\dfrac{\left|I(\ell)\right|}{d^2}+O(d(q)U_d(\ell)),
\end{equation}
where
$$
U_d(\ell):=\prod_{p\mid d}u_p(\ell).
$$
We notice further that if $(p,m)=1$, $u_p\leq 2$. Hence we have the upper bound
\begin{equation}\label{Ud<}
 U_d(\ell)\ll_m 2^{\omega(d)}.
\end{equation}
We use formula \eqref{Nd=} and the upper bound \eqref{Ud<}. So, we deduce
\begin{align}\label{before-sum-l}
\sum_{\substack{d\leq y\\(d,q)=1}}\mu(d)N_d(\ell)&=\sum_{\substack{d\leq y\\(d,q)=1}}\mu(d)\frac{\varphi(q)}{q}U_d(\ell)\dfrac{\left|I(\ell)\right|}{d^2}+O_m\left(d(q)y\log y\right)\notag\\
&=\frac{\varphi(q)}{q}\prod_{p\nmid q}\left(1-\frac{u_p(\ell)}{p^2}\right)\left|I(\ell)\right|+O_m\left(\tau(q)y\log y+Xy^{-1}\log y\right),
\end{align}
where in the second line we used that $\left|I(\ell)\right|\leq X$ and the convergence of the appearing infinite product. We proceed by realizing that if $|\ell|>\dfrac{(|m|+1)X}{q}$, the set $I(\ell)$ is empty and hence $N_d(\ell)=0$ for every $d$. This observation and equation \eqref{before-sum-l} combined give us

\begin{align}\label{lem-final}
 S_{\leq y}[m](X,q)&=\sum_{|\ell|\leq \frac{(|m|+1)X}{q}}\sum_{\substack{d\leq y\\(d,q)=1}}\mu(d)N_d(\ell)\notag\\
 &=\frac{\varphi(q)}{q}\sum_{|\ell|\leq \frac{(|m|+1)X}{q}}\prod_{p\nmid q}\left(1-\frac{u_p(\ell)}{p^2}\right)\left|I(\ell)\right|+O_m\left(\dfrac{X}{q}\left(\tau(q)y+Xy^{-1}\right)\log y\right)\notag\\
 &=\frac{\varphi(q)}{q}\sum_{\ell \in \Z}\prod_{p\nmid q}\left(1-\frac{u_p(\ell)}{p^2}\right)\left|I(\ell)\right|+O_m\left(\dfrac{X}{q}\left(\tau(q)y+Xy^{-1}\right)\log y\right).
\end{align}
We finish by a study of $u_p(\ell)$ for every $p\nmid q$. We distinguish five different cases.

\begin{itemize}
 \item If $p\mid m$, $p^2\mid \ell$ then 
 $$
 u_p(\ell)=p,
 $$
 \item If $p\mid m$, $p\mid \ell$ but $p^2\nmid \ell$ then 
 $$
 u_p(\ell)=p+1,
 $$
 \item If $p\mid m$, $p\nmid \ell$ then 
 $$
 u_p(\ell)=1,
 $$
 \item If $p\nmid m$, $p^2\mid \ell$ then 
 $$
 u_p(\ell)=1,
 $$
 \item If $p\nmid m$, $p^2\nmid \ell$ then 
 $$
 u_p(\ell)=2.
 $$ 
\end{itemize}

The lemma is now a consequence of \eqref{lem-final} and the different values of $u_p(\ell)$.
\end{proof}

Now, we can go back to the formula \eqref{S'+S''} and use \eqref{eq-lem<y} for $S_{\leq y}[m](X,q)$. What we obtain is the following

\begin{equation}\label{S'}
S[m](X,q)=\mathcal{A}[m](X,q)+S_{> y}[m](X,q)+O_m\left(\dfrac{X}{q}\left(\tau(q)y+ Xy^{-1}\right)\log y\right),
\end{equation}

where
\begin{equation}\label{A[m]}
\mathcal{A}[m](X,q)=\sum_{\ell\in \Z}f_q(\ell,m)|I(\ell)|.
\end{equation}

We finish this section by introducing the following multiplicative function for which we will deduce some properties in the next section. We let

\begin{equation}\label{h(d)}
h(d):=\mu^2(d)\prod_{p\mid d}\left(1-2p^{-2}\right)^{-1}.
\end{equation}
We point out further that we have the following equality
\begin{equation}\label{gq}
\sum_{\substack{d^2\mid \ell\\(d,r)=1}}\frac{h(d)}{d^2}=\prod_{\substack{p^2\mid \ell\\p\nmid r}}\left(\frac{p^2-1}{p^2-2}\right),
\end{equation}
for every $\ell,r\in \Z$, both nonzero.

\section{Preparation results}

In the next section we evaluate the sum $\mathcal{A}[m](X,q)$. But first we prove some preliminary results that shall be useful.

We start with a lemma that is a simplified version of \cite[Lemma 1]{Croft}.
\begin{lem}\label{lem1}
For $X>1$, $0<s<2$, and $\psi(v)$ defined by \eqref{psi},

$$\int_{0}^{X}\psi(v)v^{-s/2}dv=\dfrac{\zeta(\frac{s}{2}-1)}{(\frac{s}{2}-1)}+O\left(X^{-\frac{s}{2}}\right),$$
where the $O$-constant is absolute.
\end{lem}

\begin{proof}
We start with the formula (see \cite[equation (2.1.6)]{Titchmarsh}),

$$\zeta\left(\frac{s}{2}-1\right)=\left(\frac{s}{2}-1\right)\int_{0}^{\infty}\psi(v)v^{-s/2}dv\text{,  }(0<\sigma<2).$$
We estimate the tail of the integral

\begin{equation}\label{I=}
I=\int_{X}^{\infty}\psi(v)v^{-s/2}dv.
\end{equation}
Let $\Psi_1(v)$ be a primitive integral of $\psi(v)$ which satisfies $\Psi_1(0)=0$. So, integration by parts gives us

$$I=\Psi_1(X)X^{-\frac{s}{2}}+\dfrac{s}{2}\int_{X}^{\infty}\psi_1(v)v^{-\frac{s}{2}-1}dv.$$
Notice that since $\int_0^1\psi(v)dv=0$, $\Psi_1$ is periodic and thus bounded. Hence, we have

$$
\left|I\right|\leq 2X^{-\frac{s}{2}}
$$
which concludes the proof.
\end{proof}

The next lemma writes the function $h(d)$ as a convolution between the identity and a function which decays rapidly.

\begin{lem}\label{conv}
Let $h(d)$ be as in \eqref{h(d)}. Then we have
$$
h(d)=\sum_{d_1d_2=d}\beta(d_1),
$$
where $\beta(t)$ is supported on the cubefree numbers. Furthermore, for every cubefree number $t$, if we write $t=ab^2$ with $a,b$ squarefree, $(a,b)=1$, we have
$$
\beta(t)\ll \dfrac{d(a)}{a^2}.
$$
\end{lem}

\begin{proof}
It is easy to caculate the values of $\beta$ on prime powers
$$
\beta(p^{k})=
\begin{cases}
 1\text{, if }k=0,\\
 \frac{2}{p^2-2}\text{, if }k=1,\\
 \frac{-p^2}{p^2-2}\text{, if }k=2,\\
 0\text{, otherwise.}
\end{cases}
$$
The lemma now follows by noticing that the product
$$
\prod_{p}\left(\frac{p^2}{p^2-2}\right)
$$
is convergent.

\end{proof}
The lemma that comes next is based on \cite[Lemma 3]{Croft}. However, there are some confusing steps in the corresponding proof as it was also pointed out by Vaughan (see \cite[page 574]{Vaughan}). In particular, the value for the constant $B$ (see \cite[equation 6.3]{Croft}) is inacurate wrong. In a previous version of this paper we proceeded by a modification of Croft's arguments which led to basically the same result as stated in \cite[Lemma 3]{Croft}. More recently, inspired by a paper by Pillichshammer \cite{Pill}, we found a way of obtaining a slighlty better result by elementary means. The author later became aware of the paper \cite{Hall}, where it is implicitly contained a similar lemma with $Y^{1/3}$ replaced by $Y^{1/3+\epsilon}$ that uses again methods from complex analysis.

\begin{lem}\label{lem3}(Compare to \cite[Lemma 3]{Croft})
For $Y>0$ and $r$ an integer $\geq 1$, let

$$G(Y,r):=\sum_{(d,r)=1}h(d)\int_{0}^{Y/d^2}\psi(v)dv,$$
where $\psi(v)$ is defined by \eqref{psi} and $h(d)$ by \eqref{h(d)}.
We have, uniformly for $Y\geq 1$ and $r\neq 0$,

\begin{equation}\label{eq-aux-lem3}
\sum_{(d,r)=1}\int_{0}^{Y/d^2}\psi(v)dv=\dfrac{\varphi(r)}{r}\dfrac{\zeta(3/2)}{2\pi}Y^{1/2} + O\left(d(r)Y^{1/3}\right),
\end{equation}\\
and,
\begin{equation}\label{eq-lem3}
G(Y,r)=C'\prod_{p\mid r}\left(1+p(p^2-2)^{-1}\right)^{-1}Y^{1/2} + O\left(d(r)Y^{1/3}\right),
\end{equation}
where $C'$ is given by 

\begin{equation}\label{B-}
C'=\dfrac{\zeta\left(\frac{3}{2}\right)}{2\pi}\prod_{p}\bigg(\dfrac{p^3-3p+2}{p(p^2-2)}\bigg),
\end{equation}
and the $O$-constant is absolute.
\end{lem}

\begin{proof}

We start by proving formula \eqref{eq-aux-lem3} which is simpler. Let $D=D(Y)>0$ to be chosen later. Since $\psi(v)$ is periodic, we have

\begin{align}\label{start-aux}
\sum_{(d,r)=1}\int_{0}^{Y/d^2}\psi(v)dv&=\sum_{\substack{d\leq D\\(d,r)=1}}\int_{0}^{Y/d^2}\psi(v)dv+\sum_{\substack{d> D\\(d,r)=1}}\int_{0}^{Y/d^2}\psi(v)dv\notag\\
&=\sum_{\substack{d> D\\(d,r)=1}}\int_{0}^{Y/d^2}\psi(v)dv+O(D).
\end{align}

By inverting order of the summation and the integration, we have

\begin{align*}
\sum_{\substack{d> D\\(d,r)=1}}\int_{0}^{Y/d^2}\psi(v)dv &=\int_0^{Y/D^2}\psi(v)\sum_{\substack{D<d\leq \sqrt{\frac{Y}{v}}\\(d,r)=1}}1dv\\
&=\int_0^{Y/D^2}\psi(v)\left(\dfrac{\varphi(r)}{r}\left((Y/v\right)^{1/2}-D)+O(d(r))\right)dv\\
&=\dfrac{\varphi(r)}{r}Y^{1/2}\int_0^{Y/D^2}\psi(v)v^{-1/2}dv+O(D)+O\left(d(r)\dfrac{Y}{D^2}\right).
\end{align*}

By the lemma \ref{lem1}, the equation above implies

\begin{equation}\label{last-aux}
\sum_{\substack{d> D\\(d,r)=1}}\int_{0}^{Y/d^2}\psi(v)dv=\dfrac{-2\varphi(r)}{r}\zeta\left(-\frac{1}{2}\right)Y^{1/2}+O(D)+O\left(d(r)\dfrac{Y}{D^2}\right).
\end{equation}

We make the choice $D=Y^{1/3}$. Equation \eqref{eq-aux-lem3} is now just a consequence of \eqref{start-aux}, \eqref{last-aux} and the functional equation for the Riemann zeta function, which for $s=-1/2$ gives

$$
\zeta\left(-1/2\right)=-\frac{\zeta\left(3/2\right)}{4\pi}.
$$

We will now deduce \eqref{eq-lem3} from \eqref{eq-aux-lem3}.\\
Since 

$$
h(d)=(1\ast \beta)(d),
$$

we can write
\begin{equation}\label{G-mn}
G(Y,r)=\sum_{(d_1,r)=1}\beta(d_1)\sum_{(d_2,r)=1}\int_0^{Y/d_1^2d_2^2}\psi(v)dv.
\end{equation}

We have two possibilities. If $d_1^2\leq Y$, we shall use formula \eqref{eq-aux-lem3} for the inner sum on the right-hand side of \eqref{G-mn}. If, otherwise, $d_1^2>Y$, we shall use the trivial bound

$$
\sum_{(d_2,r)=1}\int_0^{Y/d_1^2d_2^2}\psi(v)dv\leq \sum_{d_2}\frac{Y}{d_1^2d_2^2}\ll \frac{Y}{d_1^2}.
$$
By doing so in formula \eqref{G-mn}, we obtain

\begin{multline}
G(Y,r)=\sum_{\substack{d_1\leq \sqrt{Y}\\ (d_1,r)=1}}\beta(d_1)\dfrac{\varphi(r)}{r}\dfrac{\zeta(3/2)}{2\pi}\left(\frac{Y}{d_1^2}\right)^{1/2}\\
+ O\left(d(r)\sum_{\substack{d_1\leq \sqrt{Y}\\ (d_1,r)=1}}\beta(d_1)\left(\frac{Y}{d_1^2}\right)^{1/3}+\sum_{\substack{d_1> \sqrt{Y}\\ (d_1,r)=1}}\beta(d_1)\frac{Y}{d_1^2}\right).
\end{multline}
By completing the first term and using lemma \ref{conv}, we have

\begin{multline}\label{almost-lem3}
 G(Y,r)=C_\beta(r)\dfrac{\varphi(r)}{r}\dfrac{\zeta(3/2)}{2\pi}Y^{1/2}+O\left(d(r){Y}^{1/3}\underset{ab^2\leq \sqrt{Y}}{\sum\sum}\dfrac{d(a)}{a^{\frac{8}{3}}b^{\frac{4}{3}}}\right)\\
+ O\left({Y}^{1/2}\underset{ab^2> \sqrt{Y}}{\sum\sum}\dfrac{d(a)}{a^3b^2}+Y\underset{ab^2> \sqrt{Y}}{\sum\sum}\dfrac{d(a)}{a^4b^4}\right),
\end{multline}
where
$$
C_{\beta}(r)=\sum_{(d_1,r)=1}\dfrac{\beta(d_1)}{d_1}=\prod_{p}\left(1-\dfrac{p-2}{p(p^2-2)}\right)\prod_{p\mid r}\left(1-\dfrac{p-2}{p(p^2-2)}\right)^{-1}.
$$
The first error term on the right-hand side of \eqref{almost-lem3} is clearly $\ll d(r)Y^{1/3}$. For the second one, we have
\begin{align*}
\sum_{ab^2> \sqrt{Y}}\dfrac{d(a)}{a^3b^2}&\leq \sum_{ab> Y^{1/4}}\dfrac{d(a)}{a^2b^2}\\
&\ll Y^{-1/4}(\log Y)^2,
\end{align*}
and analogously,
$$
\sum_{ab^2> \sqrt{Y}}\dfrac{d(a)}{a^4b^4}\ll Y^{-3/4}(\log Y)^2.
$$
Once we inject the above upper bounds in equation \eqref{almost-lem3}, we obtain

$$
G(Y,r)=C'\prod_{p\mid r}\left(1+p(p^2-2)^{-1}\right)^{-1}Y^{1/2}+O\left(d(r){Y}^{1/3}+Y^{1/4}(\log Y)^2\right),
$$
which concludes the proof of formula \eqref{eq-lem3}.
\end{proof}

In the following lemma we gather a series of identities that shall be useful later on, and whose proofs are routine and, hence, omitted.

\begin{lem}\label{products}

Let $m$ be a squarefree integer and let $\kappa(\rho)$ be as in \eqref{kappa}, then we have the equalities
$$
\begin{cases}
\displaystyle\sum_{\rho\cdot\sigma\mid m^2}\frac{\kappa(\rho)\mu(\sigma)}{\rho\sigma}=\displaystyle\prod_{p\mid m}\left(\frac{p^2-1}{p^2}\right),\\
\displaystyle\sum_{\rho\cdot\sigma\mid m^2}\kappa(\rho)\mu(\sigma)=\displaystyle\prod_{p\mid m}\left(\frac{p^2-p}{p^2-1}\right),\\
\displaystyle\sum_{\rho\cdot\sigma\mid m^2}\kappa(\rho)\mu(\sigma)\rho^{1/2}{\sigma}^{1/2}=\displaystyle\prod_{p\mid m}\left(\frac{p^2-p^{3/2}+p-1}{p^2-1}\right).
\end{cases}
$$ 
Also, let $r\neq 0$ be an integer, and let $h(d)$ be as in \eqref{h(d)}. Then we have further the equalities
$$
\begin{cases}
\displaystyle\sum_{(d,r)=1}\dfrac{h(d)}{d^4}=\displaystyle\prod_{p}\left(\frac{(p^2-1)^2}{p^2(p^2-2)}\right)\displaystyle\prod_{p\mid r}\left(\frac{(p^2-1)^2}{p^2(p^2-2)}\right)^{-1},\\
\displaystyle\sum_{(d,r)=1}\dfrac{h(d)}{d^2}=\displaystyle\prod_{p}\left(\frac{p^2-1}{p^2-2}\right)\displaystyle\prod_{p\mid r}\left(\frac{p^2-1}{p^2-2}\right)^{-1}.
\end{cases}
$$
Finally, let $m,q\in \Z$, $m>0$ squarefree and $q\neq 0$. Also let $f_q(\ell,m)$ be as in \eqref{fq} and $C(q)$ be as in definition \eqref{cq}. Then, we have the following equality
$$
f_q(0,m)=\frac{\varphi(mq)}{mq}C(mq)
$$
 
\end{lem}

We let
\begin{equation}\label{frakS}
\mathfrak{S}[m](Y,q):=\sum_{0<\ell\leq Y}f_q(\ell,m)(Y-\ell).
\end{equation}
The following lemma gives a a formula for $\mathfrak{S}[m](Y,q)$ by means of the previous lemmas.

\begin{prop}\label{propsum}
Let $m\neq 0$ be an integer and let $f_q(\ell,m)$ be defined by \eqref{fq}. Then, uniformly for $Y>0$ and $q$ positive integer, we have
\begin{multline}\label{propsum-eq}
\mathfrak{S}[m](Y,q)=\frac{\varphi(q)}{q}C(q)^2Y^2-\frac{\varphi(mq)}{mq}C(mq)Y\\
+\frac{C}{2}\Gamma_{\operatorname{ar}}(m)\prod_{p\mid q}(1+2p^{-1})^{-1} Y^{1/2}+O_m(d(q)Y^{1/3}),
\end{multline}
where $C,\Gamma_{\operatorname{ar}}(m)$ are as in \eqref{C-}, \eqref{Gamma-ar} respectively; $C(r)$ is as in \eqref{cq} for $r=q, mq$ and the implied $O_m-$constant depends at most on $m$.

\end{prop}

\begin{proof}
We start by recalling \eqref{fq} and the formula \eqref{gq} for last product. We have

$$
f_q(\ell,m)=C_2\prod_{p\mid m}\left(\frac{p^2-1}{p^2-2}\right)\prod_{p\mid q}\left(\frac{p^2-p}{p^2-2}\right)\kappa((\ell,m^2))\sum_{\substack{d^2\mid \ell\\(d,mq)}}\dfrac{h(d)}{d^2},
$$
where $C_2$ is as in \eqref{C_2}. We notice that the three first terms on the right-hand side of equation above are independent of $\ell$ that means that in order to evaluate $\mathfrak{S}[m](X,q)$, we need to study 
\begin{align}\label{propsum-first}
\mathfrak{S}'[m](Y,q)&:=\sum_{0<\ell\leq Y}\kappa((\ell,m^2))\sum_{\substack{d^2\mid \ell\\ (d,mq)=1}}\frac{h(d)}{d^2}\left(Y-\ell\right)\\
&=\sum_{\rho\mid m^2}\kappa(\rho)\sum_{\substack{0<\ell\leq Y\\(\ell,m^2)=\rho}}(Y-\ell)\sum_{\substack{d^2\mid \ell\\(d,mq)=1}}\dfrac{h(d)}{d^2}\notag\\
&=\sum_{\rho\sigma\mid m^2}\kappa(\rho)\mu(\sigma)\rho\sigma\sum_{(d,mq)=1}\dfrac{h(d)}{d^2}\sum_{\substack{0<\ell_0\leq \frac{Y}{\rho\sigma}\\ d^2\mid \ell_0}}\left(\frac{Y}{\rho\sigma}-\ell_0\right)
\end{align}
where in the third line we used M\"obius invertion formula for detecting the gcd condition. We proceed to write the innersum as an integral. We have

\begin{align}\label{Y->int}
\sum_{\substack{0<\ell_0<\frac{Y}{\rho\sigma}\\ d^2\mid \ell_0}}\left(\frac{Y}{\rho\sigma}-\ell_0\right)&=\sum_{\substack{0<\ell_0\leq \frac{Y}{\rho\sigma}\\ d^2\mid \ell_0}}\int_{\ell}^{\frac{Y}{\rho\sigma}}1du\notag\\
&=\int_{0}^{\frac{Y}{\rho\sigma}}\sum_{\substack{0<\ell_0\leq u\\ d^2\mid \ell_0}}1du=\int_{0}^{\frac{Y}{\rho\sigma}}\left\lfloor \frac{u}{d^2} \right\rfloor du.
\end{align}

Recall formula \eqref{psi} definining the sawtooth function $\psi$
\begin{equation*}
\lfloor x\rfloor=x-\frac{1}{2}+\psi(x).
\end{equation*}
This formula when used in equation \eqref{Y->int} gives

$$
\sum_{\substack{0<\ell_0\leq \frac{Y}{\rho\sigma}\\ d^2\mid \ell_0}}\left(\frac{Y}{\rho\sigma}-\ell_0\right)=\dfrac{Y^2}{2\rho^2\sigma^2d^2}-\dfrac{Y}{2\rho\sigma}+d^2\int_{0}^{\frac{Y}{\rho\sigma d^2}}\psi\left(v\right)dv,
$$
where, here, we calculated the integrals for the first two terms and made the change of variables $v=u/d^2$ on the last one. Injecting this formula in \eqref{propsum-first}, we deduce

\begin{equation}\label{beforeG}
\mathfrak{S}'[m](Y,q)=\lambda_2(m,q)Y^2-\lambda_1(m,q)Y+\sum_{\rho\sigma\mid m^2}\kappa(\rho)\mu(\sigma)\rho\sigma G\left(\frac{Y}{\rho\sigma},mq\right),
\end{equation}
where

\begin{equation*}
\lambda_2(m,q)=\sum_{\rho\sigma\mid m^2}\frac{\kappa(\rho)\mu(\sigma)}{\rho\sigma} \times \sum_{(d,mq)=1}\dfrac{h(d)}{d^4}
\end{equation*}
and

\begin{equation*}
\lambda_1(m,q)=\sum_{\rho\sigma\mid m^2}\kappa(\rho)\mu(\sigma) \times \sum_{(d,mq)=1}\dfrac{h(d)}{d^2}
\end{equation*}
As for $G\left(\frac{Y}{\rho\sigma},mq\right)$, there are two possibilities. Either $\frac{Y}{\rho\sigma}\geq 1$ and we can use formula \eqref{eq-lem3} giving
$$
G\left(\frac{Y}{\rho\sigma},mq\right)=C'\prod_{p\mid mq}\left(1+p(p^2-2)^{-1}\right)^{-1}\left(\frac{Y}{\rho\sigma}\right)^{1/2} + O\left(d(mq)\left(\frac{Y}{\rho\sigma}\right)^{1/3}\right).
$$
On the other hand, if $\frac{Y}{\rho\sigma}< 1$, we have both $\left(\frac{Y}{\rho\sigma}\right)^{1/2}\ll 1$ and 
$$
G\left(\frac{Y}{\rho\sigma},mq\right)\ll \sum_{d}\dfrac{Y}{\rho\sigma d^2}\ll  1.
$$
Hence we can write, in this case,
$$
G\left(\frac{Y}{\rho\sigma},mq\right)=C'\prod_{p\mid mq}\left(1+p(p^2-2)^{-1}\right)^{-1}\left(\frac{Y}{\rho\sigma}\right)^{1/2} + O\left(1\right).
$$
The above considerations now add together giving
\begin{equation}\label{sumG}
 \sum_{\rho\sigma\mid m^2}\kappa(\rho)\mu(\sigma)\rho\sigma G\left(\frac{Y}{\rho\sigma},mq\right) =\lambda(m,q)Y^{1/2}+O_m(d(q)Y^{1/3}),
\end{equation}
where
\begin{equation*}
\lambda(m,q)=C'\prod_{p\mid mq}\left(1+\frac{p}{p^2-2}\right)^{-1}\sum_{\rho\sigma\mid m^2}\kappa(\rho)\mu(\sigma)\rho^{1/2}\sigma^{1/2}.
\end{equation*}
Now, if we go back to equation \eqref{beforeG} and use \eqref{sumG} above we obtain
$$
\mathfrak{S}'[m](Y,q)=\lambda_2(m,q)Y^2-\lambda_1(m,q)Y+\lambda(m,q)Y^{1/2}+O_m(d(q)Y^{1/3}).
$$
Hence
\begin{equation}\label{Lambda}
\mathfrak{S}[m](Y,q)=\Lambda_2(m,q)Y^2-\Lambda_1(m,q)Y+\Lambda(m,q)Y^{1/2}+O_m(d(q)Y^{1/3}),
\end{equation}
where
$$
\begin{cases}

 \Lambda(m,q)=C_2\displaystyle\prod_{p\mid m}\left(\frac{p^2-1}{p^2-2}\right)\displaystyle\prod_{p\mid q}\left(\frac{p^2-p}{p^2-2}\right)\lambda(m,q),\\
 \Lambda_i(m,q)=C_2\displaystyle\prod_{p\mid m}\left(\frac{p^2-1}{p^2-2}\right)\displaystyle\prod_{p\mid q}\left(\frac{p^2-p}{p^2-2}\right)\lambda_i(m,q)\text{, }i=1,2.
 
\end{cases}
$$

Lemma \ref{products} ensures that the constants $\Lambda_2(m,q), \Lambda_1(m,q)$ and $\Lambda(m,q)$ correspond to the constants in \eqref{propsum-eq}.
The result now follows from equation \eqref{Lambda}.
\end{proof}

\section{Main term}
In this section we show how to use the results from the previous section to exhibit a formula for $\mathcal{A}[m](X,q)$, where the second main term in formula \eqref{intereq} makes its first appearance. We prove the following

\begin{prop}
For $X>1$, $m,q$ integers such that $m$ is squarefree and $q\geq 1$, let $\mathcal{A}[m](X,q)$ be defined by formula \eqref{A[m]}. Then we have, uniformly for $X,q>1$,
\begin{multline}\label{mainterm-eq}
\mathcal{A}[m](X,q)=\varphi(q)\left(C(q)\frac{X}{q}\right)^2+\frac{C}{2}\Gamma_{\operatorname{an}}(m)\Gamma_{\operatorname{ar}}(m)\prod_{p\mid q}(1+2p^{-1})^{-1}X^{1/2}q^{1/2}\\
+O_m\left(d(q)X^{1/3}q^{2/3}\right),
\end{multline}
where $C$, $C(r)$, $\Gamma_{\operatorname{an}}(m)$ and $\Gamma_{\operatorname{ar}}(m)$ are as in \eqref{C-}, \eqref{cq}, \eqref{Gamma-an} and \eqref{Gamma-ar} respectively; and the implied $O_m-$constant depends at most on $m$.
\end{prop}

\begin{proof}

There is a slighty difference depending on whether $m>0$ or $m<0$. So we study the two cases separately.

\subsection*{The case $m>0$}
We start from the formula
$$
\mathcal{A}[m](X,q)=\sum_{\ell \in \Z}f_q(\ell,m)|I(\ell)|.
$$
By analyzing the possible values of $|I(\ell)|$, we have

\begin{multline}\label{possible-ell}
\mathcal{A}[m](X,q)=f_q(0,m)\frac{X}{m}+\sum_{0<\ell\leq \frac{X}{q}}f_q(\ell,m)\left(\frac{X-\ell q}{m}\right)\\
+\sum_{-\frac{(m-1)X}{q}\leq \ell<0}f_q(\ell,m)\frac{X}{m}+\sum_{-\frac{mX}{q}\leq \ell<-\frac{(m-1)X}{q}}f_q(\ell,m)\left(\frac{X+\ell q}{m}\right).
\end{multline}

We remark that for $m,q$ fixed, $f_q(\ell,m)$ only depends on the positive divisors of $\ell$ (see formula \eqref{fq}). Hence $f_q(\ell,m)=f_q(-\ell,m)$. As a consequence, formula \eqref{possible-ell} implies

\begin{multline*}
\mathcal{A}[m](X,q)=f_q(0,m)\frac{X}{m}+\frac{1}{m}\sum_{0<\ell\leq \frac{X}{q}}f_q(\ell,m)\left(X-\ell q\right)\\
-\frac{1}{m}\sum_{0< \ell\leq \frac{(m-1)X}{q}}f_q(\ell,m)\left((m-1)X-\ell q\right)+\frac{1}{m}\sum_{0< \ell\leq \frac{mX}{q}}f_q(\ell,m)\left(mX-\ell q\right).
\end{multline*}

We recall that $\mathfrak{S}[m](Y,q)=\displaystyle\sum_{0<\ell\leq Y}f_q(\ell,m)\left(Y-\ell\right)$. So, we can write

\begin{equation}\label{before-prop}
\mathcal{A}[m](X,q)=f_q(0,m)\frac{X}{m}+\frac{q}{m}\bigg\{\mathfrak{S}[m](X/q,q)-\mathfrak{S}[m]((m-1)X/q,q)+\mathfrak{S}[m](mX/q,q)\bigg\}.
\end{equation}
By proposition \ref{propsum}, we have

$$
\mathfrak{S}[m](Y,q)=\frac{\varphi(q)}{q}C(q)^2Y^2+\frac{\varphi(mq)}{mq}C(mq)Y+\frac{C}{2}\Gamma_{\operatorname{ar}}(m)\prod_{p\mid q}(1+2p^{-1})^{-1}Y^{1/2}+O_m\left(d(q)Y^{1/3}\right).
$$
Hence, we deduce from \eqref{before-prop} that

\begin{multline*}
\mathcal{A}[m](X,q)=f_q(0,m)\frac{X}{m}+ \varphi(q)C(q)^2\left(\frac{X}{q}\right)^2-\frac{\varphi(mq)}{mq}C(mq)\frac{X}{m}\\
+\frac{C}{2}\Gamma_{\operatorname{an}}(m)\Gamma_{\operatorname{ar}}(m)\prod_{p\mid q}(1+2p^{-1})^{-1}X^{1/2}q^{1/2}+O_m\left(d(q)X^{1/3}q^{2/3}\right), 
\end{multline*}
where $\Gamma_{\operatorname{an}}(m)$ is as in \eqref{Gamma-an}. Now, by lemma \ref{products}, the first and third terms disappear. This concludes the proof of formula \eqref{mainterm-eq} in this case.

\subsection*{The case $m<0$}
Analogously to the previous case, we have
$$
\mathcal{A}[m](X,q)=\frac{q}{m}\left\{\mathfrak{S}[m](X/q,q)+\mathfrak{S}[m](-mX/q,q)-\mathfrak{S}[m]((1-m)X/q,q)\right\}
$$
and, again by proposition \ref{propsum}, we have that \eqref{mainterm-eq} is true in this case as well. The proof of the proposition is finished.

\end{proof}

\section{Bounding $S_{>y}[m](X,q)$}\label{bounding}

In the present section we give a bound for $S_{>y}[m](X,q)$. We start by noticing that $d^2\mid \xi(n;\ell q,m)$ if and only if there exist $j,k$ such that $d=jk$ and both $j^2\mid n$ and $k^2\mid mn+\ell q$. Moreover since we  are supposing $n,mn+\ell q<X$, we have $j,k<\sqrt{X}$. From this observation we deduce

\begin{align}\label{first-jk}
S_{>y}[m](X,q)&=\displaystyle\sum_{\ell \in \Z}\displaystyle\sum_{\substack{y< d\leq X\\(d,q)=1}}\mu(d)\#\{n\in I(\ell); (n,q)=1\text{ and } \xi(n;\ell q,m)\equiv 0\!\!\!\pmod d\} \notag\\
&\leq\displaystyle\sum_{\substack{j,k\leq \sqrt{X}\\jk>y\\(jk,q)=1}}\#\{(n,\ell)\in \Z^2;0<n,mn+\ell q<X\text{ and } j^2\mid n, k^2\mid mn+\ell q\}\notag\\
&=\displaystyle\sum_{\substack{j,k\leq \sqrt{X}\\jk>y\\(jk,q)=1}}N[m](X,q;j,k),
\end{align}
by definition.

We shall divide the possible values of $j$ and $k$ into sets of the form

$$\mathcal{B}(J,K):=\{(j,k); \gcd(jk,q)=1, j\sim J,k\sim K\}.$$
We can do the division using at most $O(\mathcal{L}^2)$ since we are summing over $j, k\leq X^{1/2}$.\\
Let
\begin{align}\label{NJK}
\mathcal{N}[m](J,K)=& \sum_{\substack{(jk,q)=1\\j\sim J, k\sim K}}N[m](X,q;j,k)\notag\\
=& \#\{(j,k,u,v); j\sim J,k\sim K, 0<j^2u,k^2v<X,\text{ and }mj^2u\equiv k^2v\!\!\!\pmod{q}\}
\end{align}

By taking the maximum over all $J,K$, we obtain a pair $(J,K)$ with $J,K\leq X^{1/2}$ such that

\begin{equation}\label{S''<Nlog}
S_{>y}[m](X,q)\ll \mathcal{N}[m](J,K){\mathcal{L}}^2.
\end{equation}
By the condition
$$jk>y$$
in formula \eqref{first-jk}, we can also impose

$$JK\geq \dfrac{y}{4}.$$

Our problem is now bounding $\mathcal{N}[m](J,K)$. Notice that, since we bound $S_{>y}[m](X,q)$ as in \eqref{S''<Nlog}, we will not be able to benefit from oscillations of the coefficients $\mu(d)$ in the definition \eqref{S1,S2=}. Although formula \eqref{NJK} is not symmetrical repectively to $J$ and $K$, we would like to benefit from some symmetry. With that in mind, let $m_1,m_2\in \Z$ we define

\begin{equation}\label{N12JK}
\mathcal{N}[m_1,m_2](J,K):=\#\{(j,k,u,v); j,k\in \mathcal{B}(J,K), 0<j^2u,k^2v<X,\text{ and }m_1j^2u\equiv m_2k^2v\!\!\!\pmod{q}\}.
\end{equation}
We also suppose
\begin{equation}\label{m=max}
\max(|m_1|,|m_2|)\leq |m|.
\end{equation}
In the following we estimate the general $\mathcal{N}[m_1,m_2](J,K)$ from which we can directly deduce an estimate for $\mathcal{N}[m](J,K)$ itself.
The first bound we give is an auxiliary one and will be useful later on.

\subsection{Auxiliary bound}

\begin{lem}\label{lem-aux}
Let $X,q\geq 1$, $q$ integer. Also let $m,m_1,m_2$ such that $0<m_1,m_2\leq |m|$. Let $\mathcal{N}[m_1,m_2](J,K)$ be as in \eqref{N12JK}, then for every $J,K\leq X$, we have

\begin{equation}\label{aux}
\mathcal{N}[m_1,m_2]\ll_m \dfrac{X}{q}\bigg(X(JK)^{-1}+XJ^{-2}K\bigg)\mathcal{L},
\end{equation}
where the implied constant depends at most on $m$.

\end{lem}

\begin{proof}
We start by noticing that if we make $\ell q=m_1j^2u-m_2k^2v$, we have $|\ell|\leq\frac{2|m|X}{q}$. Hence we have the inequality

$$\mathcal{N}[m_1,m_2]\leq \sum_{|\ell|\leq\frac{2|m|X}{q}}\sum_{\substack{k\sim K\\(k,q)=1}}\sum_{u\leq XJ^{-2}}\sum_{\substack{j\sim J\\ m_1j^2u\equiv -\ell q\!\!\!\pmod{k^2}\\ (j,k)=1}}1.$$

Before we go any further we must do some considerations about greatest common divisor of the variables above. First we write $f=(j,k)$. Since $(k,q)=1$, we must have $f^2\mid \ell$. We then write
$$
j_0=\frac{j}{f},k_0=\frac{k}{f}\text{ and }\ell_0=\frac{\ell}{f^2}.
$$
The congruence above then becomes
$$
m_1j_0^2u\equiv -\ell_0 q\!\!\!\pmod{k_0^2}.
$$
Now, let $g=(k_0^2,m_1)$ as above we have $g\mid \ell_0$. We write
$$
r=\frac{k^2}{g},s=\frac{m_1}{g}\text{ and }t=\frac{\ell_0}{g}.
$$
Finally, let $h=(r,t)$. From the considerations above, we must have $h\mid u$. We write
$$
r'=\frac{r}{h},t'=\frac{t}{h}\text{ and }u'=\frac{u}{h}.
$$
So the congruence becomes
$$
su'j_0^2\equiv - t'q\!\!\!\pmod{r'}
$$
and since $(t'q,r')=1$, it has at most $2.2^{\omega(r')}\leq 2d(k)$ solutions in $j_0 \pmod{t'}$. Therefore we have

\begin{align*}
\mathcal{N}[m_1,m_2](J,K)&\leq\sum_{f\geq 1}\sum_{\ell_0\leq \frac{X}{f^2q}}\sum_{\substack{k_0\sim K/f\\(k_0,q)=1}}\sum_{u'\leq XJ^{-2}h^{-1}}\sum_{\substack{j_0\sim J/f\\ su'j_0^2\equiv -t'q\!\!\!\pmod{r'}}}1\\
&\leq 2\sum_{f\geq 1}\sum_{\ell\leq \frac{X}{f^2q}}\sum_{k_0\sim K/f}XJ^{-2}h^{-1}\left\{\dfrac{Jgh}{fk_0^2}+1\right\}d(k)\\
&\ll_m \sum_{f\geq 1}\sum_{\ell\leq \frac{X}{f^2q}}\sum_{k_0\sim K/f}XJ^{-2}\left\{\dfrac{J}{fk_0^2}+1\right\}d(k)\\
&\ll_m \sum_{f\geq 1}\dfrac{X}{f^2q}XJ^{-2}\left\{\dfrac{J}{K^2}+1\right\}K\mathcal{L}\\
&\ll \dfrac{X}{q}XJ^{-2}\left\{\dfrac{J}{K^2}+1\right\}K\mathcal{L}.
\end{align*}
Hence the result.
\end{proof}

\begin{rmk}
We definetely have an inequality similar to \eqref{aux} with the roles of $J$ and $K$ interchanged. 
\end{rmk}

\subsection{Square Sieve}

We must now proceed to obtain a more precise bound. We start from the equality

$$\mathcal{N}[m_1,m_2](J,K)=\sum_{u\leq XJ^{-2}}\mathcal{N}_u[m_1,m_2](J,K),$$
where

$$\mathcal{N}_u[m_1,m_2](J,K)=\# \left\{(j,k,v); j,k\in\mathcal{B}(J,K),0<k^2v<X\text{ and }m_1j^2u\equiv m_2k^2v\!\!\!\!\pmod{q}\right\}.$$
In other words, we fix $u$ and see how many times it contributes to $\mathcal{N}[m_1,m_2](J,K)$.\\
Again by dyadic decomposition, as we did in \eqref{NJK}, we see that there is a certain $U$ satisfying

\begin{equation}\label{U<<}
U\leq XJ^{-2},
\end{equation}
such that

\begin{equation}\label{N<NU}
\mathcal{N}[m_1,m_2](J,K)\ll \mathcal{N}[m_1,m_2](J,K,U)\mathcal{L},
\end{equation}
where

$$\mathcal{N}[m_1,m_2](J,K,U)=\sum_{u\sim U}\mathcal{N}_u[m_1,m_2](J,K).$$
If we do an analysis of the possible values for $v$ and $\ell=\frac{m_2k^2v-m_1j^2u}{q}$ as we count $\mathcal{N}_u[m_1,m_2](J,K)$, we obtain

$$\mathcal{N}_u[m_1,m_2](J,K)\leq \# \left\{(j,k,\ell,v); j\sim J, k\sim K, m_1j^2u=m_2k^2v-\ell q, |\ell|\leq \frac{2|m|X}{q}, v\leq XK^{-2}, (j,k)=1\right\},$$
for every $u\sim U$.

We now appeal to the square sieve as in \cite{Heathbrown} that we state here for easier reference.

\begin{thm}\label{H-B}(\cite[Theorem 1]{Heathbrown})
Let $\mathcal{P}$ be a set of $P$ odd primes and $(w(n))_{n\geq 1}$ a sequence of real numbers. Suppose that $w(n)=0$ for $n=0$ or $n\geq e^P$. Then

$$\sum_{n\geq 1}w(n^2)\ll P^{-1}\sum_{n\geq 1}w(n)+P^{-2}\sum_{\substack{p_1\neq p_2\\ p_1,p_2\in \mathcal{P}}}\left|\sum_{n}w(n)\left(\dfrac{n}{p_1p_2}\right)\right|,$$
where $\left(\dfrac{n}{p_1p_2}\right)$ is the Jacobi symbol and the implied constant is absolute.
\end{thm}

We apply the square-sieve to the multi-set of integers

$$\mathcal{A}_u=\left\{\dfrac{m_1(m_2k^2v-\ell q)}{u}; u\mid m_2k^2v-\ell q, k\sim K, \ell\leq \dfrac{2|m|X}{q}, v\leq XK^{-2}\right\}.$$
For $n$ integer, let $w_u(n)$ denote the number of times where $n$ appears in $\mathcal{A}_u$. We have

$$\mathcal{N}_u[m_1,m_2](J,K)\leq \sum_{j\geq 1}w_u((m_1j)^2)\leq \sum_{n\geq 1}w_u(n^2).$$
For the set of primes, we take

$$\mathcal{P}=\left\{p\text{ prime }; p\nmid 2m_1m_2u\text{, }\widehat{P}<p\leq 2\widehat{P}\right\},$$
where $\widehat{P}$ will be chosen later depending on $J, K, X, q$, subject to the condition

\begin{equation}\label{condQ}
(\log X)^2\leq \widehat{P}\leq X.
\end{equation}
We have $P=\#\mathcal{P}\sim \widehat{P}(\log \widehat{P})^{-1}$ as $\widehat{P}\rightarrow \infty$ and thus

$$w_u(n)=0\text{ for }n\geq e^P,$$
because, for $X$ sufficiently large,

$$e^P\geq e^{\widehat{P}^{1/2}}\geq e^{\log X}=X,$$
and $w_u(n)=0$ for $n>X$.

Theorem \ref{H-B} and the defnition of $\mathcal{P}$ then give the inequality

\begin{equation}\label{sqsieve}
\mathcal{N}_u[m_1,m_2](J,K)\ll {\widehat{P}}^{-1}(\log X)\sum_{n\geq 1}w_u(n)+P^{-2}\sum_{\substack{p_1,p_2\\ p_1,p_2\in \mathcal{P}\\ p_1\neq p_2}}\left|\sum_{k,\ell,v}\left(\dfrac{m_1u(m_2k^2v-\ell q)}{p_1p_2}\right)\right|,
\end{equation}
where the conditions in the last sum are

\begin{equation}\label{cond}
k\sim K\text{, }\ell\leq \dfrac{X}{q}\text{, }v\leq XK^{-2}\text{, }u\mid m_2k^2v-\ell q.
\end{equation}
To simplify, we write

$$T_1(u)=\widehat{P}^{-1}(\log X)\sum_{n\geq 1}w_u(n)$$
$$T_2(u)=P^{-2}\sum_{\substack{p_1,p_2\\ p_1,p_2\in \mathcal{P}}}\left|\sum_{k,\ell,v}\left(\dfrac{u(m_2k^2v-\ell q)}{p_1p_2}\right)\right|.$$
So that \eqref{sqsieve} implies the inequality

\begin{equation}\label{T1+T2}
\mathcal{N}[m_1,m_2](J,K,U)\ll \sum_{u\sim U}T_1(u)+\sum_{u\sim U}T_2(u).
\end{equation}

\subsubsection{Study of $T_1(u)$}
Because of the definition of $w_u(n)$, we have

\begin{align*}
\sum_{u\sim U}T_1(u)&\ll \widehat{P}^{-1}(\log X)\sum_{k,\ell,v}\sum_{u\mid m_2k^2v-\ell q}	1\\
&\ll \widehat{P}^{-1}(\log X)\sum_{k,\ell,v}d(m_2k^2v-\ell q),
\end{align*}
where the conditions on the sum are

\begin{equation}\label{cond-u}
k\sim K\text{, }\ell\leq \dfrac{X}{q}\text{, }v\leq XK^{-2}\text{, }m_2k^2v-\ell q\geq 1.
\end{equation}
For the sum above, we appeal to Lemma \ref{sumdivfunc} below, giving the inequality

\begin{equation}\label{T1}
\sum_{u\sim U}T_1(u)\ll_{\eta} \frac{X}{q}(X^{1/2+\eta}+X(K\widehat{P})^{-1}\mathcal{L}^3),
\end{equation}
for every $\eta>0$, where the implied constant depends only on $\eta$.\\

Note that for the first term of \eqref{T1}, we used the inequality $\widehat{P}^{-1}\log X\leq 1$ (see \eqref{condQ}).

\subsubsection{Study of $T_2(u)$}

We have

\begin{align}\label{max}
\sum_{u\sim U}T_2(u)&=P^{-2}\sum_{p_1\neq p_2}\sum_{u\sim U}\left|\sum_{k,\ell,v}\left(\dfrac{m_1u(m_2k^2v-\ell q)}{p_1p_2}\right)\right| \notag\\
&\leq \sum_{u\sim U}\max_{\substack{p_i\nmid 2m_1m_2u\\ \widehat{P} < p_1 < p_2\leq 2\widehat{P}}}\left|\sum_{k,\ell,v}\left(\dfrac{(m_2k^2v-\ell q)}{p_1p_2}\right)\right|.
\end{align}
We pick $(p_1, p_2)=(p_1(u),p_2(u))$ for which the maximum is attained, and we proceed to estimate the sum

\begin{equation}\label{SSS}
S_u:=\sum_{k,\ell,v}\left(\dfrac{m_2k^2v-\ell q}{p_1p_2}\right),
\end{equation}
where the conditions on $k$, $\ell$, $v$ are as in \eqref{cond}. And we recall $p_1\neq p_2$.

Thus, \eqref{max} becomes
\begin{equation}\label{sum-T_2<<S}
\sum_{u\sim U}T_2(u)\ll \sum_{u\sim U}\left\vert S_u\right\vert
\end{equation}
The following calculations follow the same lines of the main calculations in the proof of \cite[Theorem 2]{Heathbrown}. The main difference is the additional sum over $\ell$ which induces big cancellations.
We write

\begin{align}\label{S=exp}
S_u=&\sum_{\substack{\alpha,\beta,\gamma=1\\ u\mid m_2\alpha^2\beta-q\gamma}}^{up_1p_2}\left(\dfrac{\alpha^2\beta-q\gamma}{p_1p_2}\right)\left|\sum_{\substack{k\sim K\\ k\equiv \alpha\!\!\!\pmod{up_1p_2}}}1\right|\left|\sum_{\substack{v\leq XK^{-2}\\ v\equiv \beta\!\!\!\pmod{up_1p_2}}}1\right|\left|\sum_{\substack{\ell\leq X/q\\ \ell\equiv \gamma\!\!\!\pmod{up_1p_2}}}1\right|\notag\\
=&\sum_{\substack{\alpha,\beta,\gamma\\ u\mid m_2\alpha^2\beta-q\gamma}}\left(\dfrac{\alpha^2\beta-q\gamma}{p_1p_2}\right)\left\{\dfrac{1}{up_1p_2}\sum_{\lambda=1}^{up_1p_2}\sum_{k\sim K}e\left(\dfrac{\lambda(\alpha-k)}{up_1p_2}\right)\right\}\left\{\dfrac{1}{up_1p_2}\sum_{\mu=1}^{up_1p_2}\sum_{v\leq XK^{-2}}e\left(\dfrac{\mu(\beta-v)}{up_1p_2}\right)\right\}\notag\\
&\cdot\left\{\dfrac{1}{up_1p_2}\sum_{\nu=1}^{up_1p_2}\sum_{\ell\leq \frac{X}{q}}e\left(\dfrac{\nu(\gamma-\ell)}{up_1p_2}\right)\right\}\notag\\
=&(up_1p_2)^{-3}\sum_{\lambda,\mu,\nu=1}^{up_1p_2}S(u,p_1p_2,q,m_2;\lambda,\mu,\nu)\Theta_{\lambda}\Phi_{\mu}\Psi_{\nu},
\end{align}
where

\begin{equation}\label{BoundsPhimu}
\begin{cases}
S(u,p_1p_2,q,m_2;\lambda,\mu,\nu)=\displaystyle\sum_{\substack{\alpha,\beta,\gamma=1\\ u\mid m_2\alpha^2\beta-q\gamma}}^{up_1p_2}\left(\dfrac{\alpha^2\beta-q\gamma}{p_1p_2}\right)e\left(\dfrac{\lambda\alpha+\mu\beta+\nu\gamma}{up_1p_2}\right),\\
\Theta_{\lambda}=\displaystyle\sum_{K<k\leq 2K}e\left(\dfrac{-\lambda k}{up_1p_2}\right)\ll \min\left(K,\left\|\dfrac{\lambda}{up_1p_2}\right\|^{-1}\right),\\
\Phi_{\mu}=\displaystyle\sum_{v\leq XK^{-2}}e\left(\dfrac{-\mu v}{up_1p_2}\right)\ll \min\left(XK^{-2},\left\|\dfrac{\mu}{up_1p_2}\right\|^{-1}\right),\\
\Psi_{\nu}=\displaystyle\sum_{\ell\leq \frac{X}{q}}e\left(\dfrac{-\nu \ell}{up_1p_2}\right)\ll \min\left(\dfrac{X}{q},\left\|\dfrac{\nu}{up_1p_2}\right\|^{-1}\right),\\
\end{cases}
\end{equation}
where $\|x\|$ denotes the distance from $x$ to the nearest integer.\\
The chinese remainder theorem allows us to write

\begin{equation}\label{multiply-S}
S(u,p_1p_2;\lambda,\mu,\nu)=S_1(p_1;b,c,d)S_1(p_2;b,c,d)\prod_{r^f\|u}S_2(r^f;b,c,d), 
\end{equation}
where $u=\prod r^f$ is the prime factorization of $u$, and $b,c,d$ are some integers such that

$$
\begin{cases}
(b,up_1p_2)=(\lambda,up_1p_2),\\
(c,up_1p_2)=(\mu,up_1p_2),\\
(d,up_1p_2)=(\nu,up_1p_2)
\end{cases}
$$
and

$$
\begin{cases}
S_1(p,q,m_2;b,c,d)=\displaystyle\sum_{\alpha,\beta,\gamma=1}^{p}\left(\dfrac{m_2\alpha^2\beta-q\gamma}{p}\right)e\left(\dfrac{b\alpha+c\beta+d\gamma}{p}\right)\\
S_2(r^f,q,m_2;b,c,d)=\displaystyle\sum_{\substack{\alpha,\beta,\gamma=1\\ r^f\mid m_2\alpha^2\beta-q\gamma}}^{r^f}e\left(\dfrac{b\alpha+c\beta+d\gamma}{r^f}\right).
\end{cases}
$$
For these sums we have the following bounds whose proofs are elementary and we postpone until section \ref{app} (see Lemma \ref{expsums}).

$$
\begin{cases}
S_1(p,q,m_2;b,c,0)= 0,\\
S_1(p,q,m_2;b,c,d)\ll p^{3/2},\\
\vert S_2(r,q,m_2;b,c,d)\vert\leq 2r(r,b,c,dm_2)\\
\vert S_2(r^f,q,m_2;b,c,d)\vert\leq 2r^{\frac{3f}{2}}(r^f,b,c,dm_2)\text{, if $r$ is odd, $f\geq 2$},\\
\vert S_2(2^f,q,m_2;b,c,d)\vert\leq 4.2^{\frac{3f}{2}}(2^f,b,c,dm_2)\text{, $f\geq 2$}.
\end{cases}
$$
If we multiply these upper bounds in \eqref{multiply-S}, we have, whenever $d\not\equiv 0\!\!\!\pmod{up_1p_2}$,

\begin{align}\label{mult-bound}
S(u,p_1p_2,m_2;b,c,d)&\ll d(u)\widehat{P}^3U^{3/2}(u^{\dag})^{-1/2}(u,b,c,dm_2)\notag\\
&\ll_{m} d(u)\widehat{P}^3U^{3/2}(u^{\dag})^{-1/2}(u,b,c,d),
\end{align}
where 
\begin{equation}\label{udag}
u^{\dag}=\prod_{\substack{p\mid u\\ p^2\nmid u}}p. 
\end{equation}
and
$$
S(u,p_1p_2,m_2;b,c,0)=0.
$$
Hence, by \eqref{S=exp}, we deduce the inequality

\begin{equation}\label{S<<ThPhPs}
S_u\ll_{m} d(u)\widehat{P}^{-3}U^{-3/2}(u^{\dag})^{-1/2}\sum_{\substack{\lambda,\mu,\nu\!\!\!\!\pmod{up_1p_2}\\\nu\not\equiv 0\!\!\!\!\pmod{up_1p_2}}}|\Theta_{\lambda}\Phi_{\mu}\Psi_{\nu}|(u,\lambda,\mu,\nu).
\end{equation}

Now, we separate the above triple sum in eight parts accordingly to whether $\lambda$, $\mu$ and $\nu$ are zero or not. We also use the bounds coming from \eqref{BoundsPhimu}. We use the first term inside the $\min$-symbol in the zero case and the second one for the others. We then have

\begin{multline}\label{firstbigsum}
\sum_{\lambda,\mu,\nu\!\!\!\!\pmod{up_1p_2}}|\Theta_{\lambda}\Phi_{\mu}\Psi_{\nu}|(u,\lambda,\mu,\nu)\ll K.XK^{-2}.U\widehat{P}^2\sum_{1\leq\nu\leq\frac{up_1p_2}{2}}\nu^{-1}(u,\nu)+K(U\widehat{P}^2)^2\sum_{1\leq\mu,\nu\leq\frac{up_1p_2}{2}}\mu^{-1}\nu^{-1}(u,\mu,\nu)\\
+XK^{-2}(U\widehat{P}^2)^2\sum_{1\leq\lambda,\nu\leq\frac{up_1p_2}{2}}\lambda^{-1}\nu^{-1}(u,\lambda,\nu)+(U\widehat{P}^2)^3\sum_{1\leq\lambda,\mu,\nu\leq\frac{up_1p_2}{2}}\lambda^{-1}\mu^{-1}\nu^{-1}(u,\lambda,\mu,\nu).
\end{multline}
For the sums involving greatest common divisors, we have the following Lemma

\begin{lem}
Let $Z\geq 1$ and $u$ an integer such that $u\geq 1$, then we have the following inequalities

\begin{equation}\label{Z1}
\sum_{1\leq \lambda\leq Z}\lambda^{-1}(u,\lambda)\leq d(u)\left(\log (2Z)\right),
\end{equation}
\begin{equation}\label{Z2}
\sum_{1\leq \lambda,\mu\leq Z}\lambda^{-1}\mu^{-1}(u,\lambda,\mu)\leq d(u)\left(\log (2Z)\right)^2,
\end{equation}
\begin{equation}\label{Z3}
\sum_{1\leq \lambda,\mu,\nu\leq Z}\lambda^{-1}\mu^{-1}\nu^{-1}(u,\lambda,\mu,\nu)\ll \left(\log 2Z)\right)^3.
\end{equation}
\end{lem}

\begin{proof}
We prove \eqref{Z3}, the other ones are analogous. We have the following

\begin{align*}
\sum_{1\leq \lambda,\mu,\nu\leq Z}\lambda^{-1}\mu^{-1}\nu^{-1}(u,\lambda,\mu,\nu)&\leq \sum_{d\mid u}d\sum_{\substack{1\leq \lambda,\mu,\nu\leq Z\\ d\mid \lambda,\mu,\nu}}\lambda^{-1}\mu^{-1}\nu^{-1}\\
&\leq \sum_{d\mid u}d\sum_{1\leq \lambda',\mu',\nu'\leq \frac{Z}{d}}d^{-3}\lambda'^{-1}\mu'^{-1}\nu'^{-1}\\
&=\sum_{d\mid u}d^{-2}\left(\sum_{1\leq \lambda'\leq \frac{Z}{d}}\lambda'^{-1}\right)\\
&\ll \left(\log (2Z)\right)^3
\end{align*}

\end{proof}

If we insert the inequalities \eqref{Z1}, \eqref{Z2} and \eqref{Z3} with $Z=\frac{up_1p_2}{2}$ in \eqref{firstbigsum} and notice that $\log (up_1p_2)\leq 3\mathcal{L}$, we deduce

\begin{equation}\label{before-udag}
S_u\ll_m d(u)^2(u^{\dag})^{-1/2}\mathcal{L}^3\left\{K^{-1}\widehat{P}^{-1}U^{-1/2}X+K\widehat{P}U^{1/2}+K^{-2}\widehat{P}U^{1/2}X+\widehat{P}^{3}U^{3/2}\right\}.
\end{equation}
in order to compute the contribution of $S$ to $\mathcal{N}[m_1,m_2](J,K,U)$ we must sum over $u$. To do so, we present the following

\begin{lem}\label{sum-udag}
Let $U\geq 1$ be a real number. For each $u$ integer, such that $u\geq 1$, let $u^{\dag}$ be its squarefree part as in \eqref{udag}. Then we have

\begin{equation}\label{eq-sum-udag}
\sum_{u\sim U}d(u)^2(u^{\dag})^{-1/2}\ll U^{1/2}\left(\log (2U)\right)^3,
\end{equation}
where the implied constant is absolute.
\end{lem}

\begin{proof}
It is sufficient to prove that

$$\sum_{u\leq U}d(u)^2(u^{\dag})^{-1/2}u^{1/2}\ll U(\log U)^3.$$
Considering the Dirichlet series

$$F(s)=\sum_{n\geq 1}d(n)^2(n^{\dag})^{-1/2}n^{1/2}n^{-s}$$
and expressing it in an Euler product, we see that $F(s)\zeta^{-4}(s)$ is holomorphic in the half plane

$$\operatorname{Re}(s)>\sigma\text{, for some }\sigma<1.$$
Hence, $F(s)$ has a pole of order at most 4 at $s=1$. this gives the result.
\end{proof}

Gathering \eqref{sum-T_2<<S}, \eqref{before-udag}, and summing over $u$ by Lemma \ref{sum-udag}, we obtain

\begin{equation*}
\sum_{u\sim U}T_2(u)\ll_m \left\{K^{-1}\widehat{P}^{-1}X+K\widehat{P}U+K^{-2}\widehat{P}UX+\widehat{P}^3U^2\right\}\mathcal{L}^6
\end{equation*}
Appealing to \eqref{U<<}, we have the inequality

\begin{equation}\label{T2}
\sum_{u\sim U}T_2(u)\ll_m \left\{K^{-1}\widehat{P}^{-1}X + J^{-2}K\widehat{P}X  + J^{-2}K^{-2}\widehat{P} X^2 + J^{-4}\widehat{P}^3X^2\right\}{\mathcal{L}}^6.
\end{equation}
From \eqref{T1+T2}, \eqref{T1} and \eqref{T2}, we deduce an inequality for $\mathcal{N}[m_1,m_2](J,K,U)$

\begin{align*}
\mathcal{N}[m_1,m_2](J,K,U)\ll_{\eta,m} &\left\{X^{3/2+\eta}q^{1} + K^{-1}\widehat{P}^{-1}X^{2}q^{-1}+ K^{-1}\widehat{P}^{-1}X  \right.\notag\\
&+ J^{-2}K\widehat{P}X\left. + J^{-2}K^{-2}\widehat{P} X^2 + J^{-4}\widehat{P}^3X^2\right\}{\mathcal{L}}^6.
\end{align*}
Since $\frac{X}{q}\geq 1$, the third term above can be seen to be

$$\leq K^{-1}\widehat{P}^{-1}X^2q^{-1}.$$
And therefore,

\begin{align}\label{last-m1m2}
\mathcal{N}[m_1,m_2](J,K,U)\ll_{\eta,m} &\left\{X^{3/2+\eta}q^{1} + K^{-1}\widehat{P}^{-1}X^{2}q^{-1} + J^{-2}K\widehat{P}X\right.\notag\\
&\left. + J^{-2}K^{-2}\widehat{P} X^2 + J^{-4}\widehat{P}^3X^2\right\}{\mathcal{L}}^6.
\end{align}

\begin{rmk}
We recall the importance of working with the general $\mathcal{N}[m_1,m_2](J,K)$ instead of $\mathcal{N}[m](J,K)$. The reason is that because of the equalities
$$
\mathcal{N}[m](J,K)=\mathcal{N}[m,1](J,K)=\mathcal{N}[1,m](K,J),
$$
the upper bound \eqref{last-m1m2} still holds if we replace the roles of $J$ and $K$. in other words, there is no loss of generality in supposing $J\geq K$.
\end{rmk}

Now, the upper bound \eqref{last-m1m2}, together with \eqref{S''<Nlog} and \eqref{N<NU} gives us

\begin{align}\label{Sd''again}
S_{>y}[m](X,q)\ll_{\eta,m} &\left\{X^{3/2+\eta}q^{1} + K^{-1}\widehat{P}^{-1}X^{2}q^{-1} + J^{-2}K\widehat{P}X\right.\notag\\
&\left. + J^{-2}K^{-2}\widehat{P} X^2 + J^{-4}\widehat{P}^3X^2\right\}{\mathcal{L}}^9.
\end{align}
At this point we make the choice

\begin{equation}\label{def-phat}
\widehat{P}=JK^{-1/4}q^{-1/4} + \mathcal{L}^2, 
\end{equation}
which makes the second and the last terms of \eqref{Sd''again} similar. Hence we have

\begin{align*}
S_{>y}[m](X,q)\ll_{\eta} &\left\{X^{3/2+\eta}q^{-1} + J^{-1}K^{-3/4}X^2q^{-3/4} + J^{-1}K^{3/4}Xq^{-1/4}\right.\\
&+\left. J^{-1}K^{-9/4}X^2q^{-1/4} +J^{-2}KX + J^{-2}K^{-2}X^2 + J^{-4}X^2\right\}{\mathcal{L}}^{15}.
\end{align*}
Since $J\geq K$, $JK\geq y$, we have

\begin{align}\label{Sd''last}
S_{>y}[m](X,q)\ll_{\eta} &\left\{X^{3/2+\eta}q^{-1} + X^2q^{-3/4}y^{-7/8} + Xq^{-1/4}y^{-1/8}\right.\notag\\
&+\left.J^{-1}K^{-9/4}X^2q^{-1/4} + Xy^{-1/2}+ X^2y^{-2} \right\}{\mathcal{L}}^{15}.
\end{align}

\section{Proof of Theorem \ref{intermediate}}\label{proof-of-int}

We assume
\begin{equation}\label{y>sqrtX}
y\geq X^{\delta}\text{, for some }\delta >1/2, 
\end{equation}
and we choose $\eta$ sufficiently small so that we have
$$
\eta<\delta-\dfrac{1}{2}.
$$
We go back to \eqref{S'+S''}. Since we have studied $S_{\leq y}[m](X,q)$ (see \eqref{S'}) and $S_{< y}[m](X,q)$ (see \eqref{Sd''last}), we deduce

\begin{equation}\label{R}
S[m](X,q)=\mathcal{A}[m](X,q)+R[m](X,q)
\end{equation}

where $R[m](x,q)$ satisfies
\begin{align*}
R[m](X,q)\ll_m  &\left\{d(q)Xyq^{-1}+X^2y^{-7/8}q^{-3/4}+Xy^{-1/8}q^{-1/4}\right.\\
&+ \left.  J^{-1}K^{-9/4}X^2q^{-1/4}+Xy^{-1/2} + X^2y^{-2}\right\}{\mathcal{L}}^{15},\\
\end{align*}
and the missing term can be seen to be $\ll d(q)Xyq^{-1}\mathcal{L}^{15}$ thanks to \eqref{y>sqrtX}.
We now make the choice

\begin{equation}\label{def-y}
y=X^{8/15}q^{2/15} 
\end{equation}
to make the two first terms similar.\\
Notice that this choice of $y$ implies that \eqref{y>sqrtX} is satisfied. Hence we have

\begin{align}\label{beforeaux}
R[m](X,q)\ll_m & \left\{d(q)X^{23/15}q^{-13/15} + X^{14/15}q^{-4/15} + J^{-1}K^{-9/4}X^2q^{-1/4}\right.\notag\\
&+ \left.X^{11/15}q^{-1/15} + X^{14/15}q^{-4/15}\right\}\mathcal{L}^{15}\notag\\
\ll & \left\{d(q)X^{23/15}q^{-13/15} + \mathcal{E}_1\right\}\mathcal{L}^{15},
\end{align}
where

$$\mathcal{E}_1=X^2J^{-1}K^{-9/4}q^{-1/4}$$
and the missing terms are $\leq d(q)X^{23/15}q^{-13/15}$, since $q\leq X$.\\

At this point, we recall the auxiliary bound from lemma \ref{lem-aux}. The inequality \eqref{aux}, together with \eqref{S'+S''}, \eqref{S'}, \eqref{S''<Nlog} and, of course, the choice $y=X^{8/15}q^{2/15}$ give us

\begin{align}\label{auxfin}
\!\!\!\!R[m](X,q)\ll_m &\left\{d(q)Xyq^{-1} + X^2J^{-1}K^{-1}q^{-1} + X^2J^{-2}Kq^{-1}\right\}{\mathcal{L}}^3\notag\\
\ll &\left\{d(q)Xyq^{-1} + X^2y^{-1}q^{-1} + \mathcal{E}_2\right\}{\mathcal{L}}^{3}\notag\\
\ll &\left\{d(q)X^{23/15}q^{-13/15} + \mathcal{E}_2\right\}\mathcal{L}^{3},
\end{align}
where, again the missing term can be seen to be $\ll d(q)Xyq^{-1}{\mathcal{L}}^{3}$ thanks to \eqref{y>sqrtX} and

$$\mathcal{E}_2=X^2J^{-2}Kq^{-1}.$$
The importance of the auxiliary bound \eqref{aux} is that since $J\geq K$, we cannot give a good estimate for the term $\mathcal{E}_1$ using that $JK\geq y$. So we look for a mixed term that lies between $\mathcal{E}_1$ and $\mathcal{E}_2$ for which we have a good bound (much better than the one for $\mathcal{E}_2$, for example).\\
Notice that
\begin{align}\label{minpowers}
\min(\mathcal{E}_1,\mathcal{E}_2)\leq& \mathcal{E}_1^{12/17}\mathcal{E}_2^{5/17}\notag\\
=&X^2(JK)^{-22/17}q^{-8/17}\notag\\
\ll& X^2y^{-22/17}q^{-8/17}\notag\\
\ll& X^{334/255}q^{-164/255}.
\end{align}
Now, \eqref{beforeaux}, \eqref{auxfin}, \eqref{minpowers} imply

\begin{align}\label{R-end}
R[m](X,q)\ll_m &\left\{d(q)X^{23/15}q^{-13/15}+ X^{334/255}q^{-164/255}\right\}{\mathcal{L}}^{15}\notag\\
\ll &d(q)X^{23/15}q^{-13/15}{\mathcal{L}}^{15},
\end{align}
where the missing term is $\leq d(q)X^{23/15}q^{-13/15}$ since $q\leq X$.

The formula \eqref{intereq} now follows from \eqref{mainterm-eq} and the bound \eqref{R-end} for $R[m](X,q)$ above.\\
The proof of Theorem \ref{intermediate} is complete.

\section{Proof of Theorem \ref{main}}\label{proof-of-main}

In this section we prove \eqref{maineq} as a corollary of theorem \ref{intermediate}.\\
If we take a look at \eqref{developV}, we see that the only term we still need to estimate is

$$\sum_{\substack{n\leq X\\ (n,q)=1}}\mu^2(n).$$
The following is a very simple case of \cite[Lemma 3.1.]{Blomer}.

\begin{lem}\label{lem-bl}
Let $\epsilon>0$ be given. We then have the equality

\begin{equation}\label{lemmabl}
\sum_{\substack{n\leq X\\ (n,q)=1}}\mu^2(n)=\dfrac{\varphi(q)}{q}\dfrac{6}{\pi^2}\prod_{p\mid q}\left(1-\frac{1}{p^2}\right)^{-1}X+O\left(X^{1/2+\epsilon}\right),
\end{equation}
uniformly for $1\leq q\leq X$, where the $O$-constant depends on $\epsilon$ alone.
\end{lem}

We now prove Theorem \ref{main} by substituting equation \eqref{lemmabl} in \eqref{developV}. We obtain the equality

\begin{align*}
\mathcal{M}_2[m](X,q)=&2S[m](X,q)-\varphi(q)\left(C(q)\dfrac{X}{q}\right)^2+O(X^{3/2+\epsilon}q^{-1}).
\end{align*}
Now, we use \eqref{intereq} for the term $S[m](X,q)$ and the main term disappears. So we deduce

\begin{multline*}
\mathcal{M}_2[m](X,q)=\frac{C}{2}\Gamma_{\operatorname{an}}(m)\Gamma_{\operatorname{ar}}(m)\prod_{p\mid q}\bigg(1+2p^{-1}\bigg)^{-1}X^{1/2}q^{1/2}\\
+O_{\epsilon,m}(d(q)X^{1/3}q^{2/3}+d(q)X^{23/15}q^{-13/15}{\mathcal{L}}^{15}+X^{3/2+\epsilon}q^{-1}).
\end{multline*}
By comparing the error terms, if we choose $\epsilon$ small enough($\epsilon<1/30$ suffices), then the last error term is smaller than the second one. This concludes the proof of \eqref{maineq}.

${}$\\

\section{Appendix}\label{app}

\subsection{Sums involving the divisor function}\label{sumd}

In this appendix we analyze a certain sum involving the divisor function which we used in the corresponding proof. For that purpose, we use the following classical Lemma that can be found in \cite{Shiu}, for instance.

\begin{lem}\label{shiu}
Let $a,q\in\mathbb{Z}$, $q>0$ such that $(a,q)=1$, let $X>0$ and $\eta>0$ be given. Then

$$\sum_{\substack{n\leq X\\ n\equiv a\!\!\!\pmod{q}}}d(n)\ll_{\eta} \dfrac{\varphi(q)}{q^2}X(\log 2X),$$
uniformly for $q\leq X^{1-\eta}$. The implied constant in $\ll_{\eta}$ depending on $\eta$ alone.
\end{lem}
We need to estimate the sum

$$\sum_{k, v, \ell}d(k^2v-\ell q),$$
where the conditions on the sum are as in \eqref{cond-u}.\\
For that we have the more general result

\begin{lem}\label{sumdivfunc}
Let $q\geq 1$ be an integer and let $\eta>0$ be given. Then uniformly for $K, S, X\geq 1$ real numbers such that  $S\leq XK^{-2}$, we have

$$\sum_{k\sim K}\sum_{\ell\leq \frac{X}{q}}\sum_{\substack{v\leq S\\ k^2v-\ell q\geq 1}}d(k^2v-\ell q)\ll_{\eta} \dfrac{X}{q}\left(X^{1/2+\eta}+XK^{-1}\mathcal{L}^3\right),$$
where the implied constant depends only on $\eta$.
\end{lem}

\begin{proof}

Let $f=(k^2,\ell q)$. We write $k^2=fg$, $\ell q=fh$.\\

Let $\eta>0$. We divide in two cases
\begin{enumerate}[i)]
\item If $K\leq \dfrac{1}{2}X^{1/2-\eta/2}$, then $k^2\leq X^{1-\eta}$ and also $\frac{k^2}{f}\leq \frac{X^{1-\eta}}{f}\leq\left(\frac{4X}{f}\right)^{1-\eta}$. So by Lemma \ref{sumdivfunc} above,

\begin{align*}
\sum_{\substack{v\leq S\\ k^2v-\ell q\geq 1}}d(k^2v-\ell q)&\leq \sum_{v\leq S}d(f)d(gv-h)\\
&\ll d(f)\sum_{\substack{n\leq \frac{k^2S}{f}\\ n\equiv -h\!\!\!\pmod{g}}}d(n)\\
&\ll d(f)\sum_{\substack{n\leq \frac{4X}{f}\\ n\equiv -h\!\!\!\pmod{g}}}d(n)\\
&\ll_{\eta} d(f)\varphi(g)\frac{X}{fg^2}\mathcal{L}. 
\end{align*}
So, since $\varphi(g)\leq g$, $f\mid k^2$ and $fg=k^2$, we obtain

\begin{equation}\label{A1}
\sum_{\substack{v\leq S\\ k^2v-\ell q\geq 1}}d(k^2v-\ell q)\ll_{\eta} d(k^2)XK^{-2}\mathcal{L}.
\end{equation}

\item If $K>\frac{1}{2}X^{1/2-\eta/2}$, we have that $X/K \leq 2X^{1/2+\eta/2}$.\\
Then, by the classical bound $d(n)\ll_{\eta} n^{\eta}$, we obtain

\begin{align}\label{A2}
\sum_{\substack{v\leq S\\ k^2v-\ell q\geq 1}}d(k^2v-\ell q)&\ll_{\eta}S(k^2v-\ell q)^{\eta/2}\notag\\
&\ll \dfrac{X}{K^2}X^{\eta/2}\notag\\
&\ll X^{1/2+\eta}K^{-1}.
\end{align}
\end{enumerate}
Now, if we put the two cases together, \eqref{A1} and \eqref{A2} lead to

$$\displaystyle\sum_{\substack{v\leq S\\ k^2v-\ell q\geq 1}}d(k^2v-\ell q)\ll_{\eta} X^{1/2+\eta}K^{-1} + d(k^2)XK^{-2}\mathcal{L}.$$
After summation over $k$ and $\ell$, we finally obtain

\begin{align*}
\sum_{k, \ell, v}d(k^2v-\ell q)&\ll_{\eta} \sum_{k\sim K}\sum_{\ell\leq \frac{X}{q}}\left(X^{1/2+\eta}K^{-1}+d(k^2)XK^{-2}\mathcal{L}\right)\\
&\ll_{\eta} \frac{X}{q}\left(X^{1/2+\eta}+XK^{-1}\mathcal{L}^3\right).
\end{align*}

\end{proof}

\subsection{Exponential Sums}

Here we give the proof for the bounds on exponential sums. We have the following

\begin{lem}\label{expsums}
Let $p,q,m_2,b,c,d$ be integers, $m_2\neq 0$, $p$ a prime number and $p\nmid m_2q$. We define

$$S_1(p,q,m_2;b,c,d)=\sum_{\alpha,\beta,\gamma=1}^{p}\left(\dfrac{m_2\alpha^2\beta-q\gamma}{p}\right)e\left(\dfrac{b\alpha+c\beta+d\gamma}{p}\right).$$
Let $r$ be a prime number such that $r\nmid q$ and $f>0$ an integer,

$$S_2(r^f,q,m_2;b,c,d)=\sum_{\substack{\alpha,\beta,\gamma=1\\ r^f\mid m_2\alpha^2\beta-q\gamma}}^{r^f}e\left(\dfrac{b\alpha+c\beta+d\gamma}{r^f}\right).$$
Then, we have

\begin{numcases}{}
S_1(p,q,m_2;b,c,0)=0,\label{zero}\\
S_1(p,q,m_2;b,c,d)\ll p^{3/2},\label{A}\\
|S_2(r,q,m_2;b,c,d)|\leq 2r(r,b,c,dm_2)\label{B},\\
|S_2(r^f,q,m_2;b,c,d)|\leq 2r^{\frac{3f}{2}}(r^f,b,c,dm_2)^{\frac{1}{2}}\text{ if $r$ is odd, $f\geq 2$}\label{C},\\
|S_2(2^f,q,m_2;b,c,d)|\leq 4.2^{\frac{3f}{2}}(2^f,b,c,dm_2)^{\frac{1}{2}},\text{ $f\geq 2$}\label{D}.
\end{numcases}

\end{lem}

\begin{proof}

We start by the study of $S_2$.\\
We have

\begin{align}\label{delta}
S_2(r,q,m_2;b,c,d)&=\displaystyle\sum_{\alpha,\beta=1}^{r}e\left(\dfrac{b\alpha+c\beta+d\overline{q}m_2\alpha^2\beta}{r}\right)\notag\\
&=\displaystyle\sum_{\alpha=1}^{r}e\left(\dfrac{b\alpha}{r}\right)\displaystyle\sum_{\beta=1}^{r}e\left(\dfrac{(c+dm_2\overline{q}\alpha^2)\beta}{r}\right)\notag\\
&=\displaystyle\sum_{\alpha=1}^{r}e\left(\dfrac{b\alpha}{r}\right)\delta(c+dm_2\overline{q}\alpha^2,r),
\end{align}
where

$$
\delta(x,n)=
\begin{cases}
n \text{ if $n\mid x$,}\\
0 \text{ otherwise.}
\end{cases}
$$
If $r\mid dm_2$, equation \eqref{delta} becomes

$$S_2(r,q,m_2;b,c,d)=\displaystyle\sum_{\alpha=1}^{r}e\left(\dfrac{b\alpha}{r}\right)\delta(c,r)=\delta(b,r)\delta(c,r).$$
So, $S_2(r,q;b,c,d)=0$ unless $r$ divides both $b$ and $c$, in which case,

$$S_2(r,q,m_2;b,c,d)=r^2=r(r,b,c,dm_2).$$
That proves \eqref{B} when $r\mid dm_2$.	

We now assume $r\nmid dm_2$. We analyze when the symbol $\delta(c+dm_2\overline{q}\alpha^2)$ is non-zero. That means we consider the equation

$$dm_2\alpha^2\equiv -cq\!\!\!\!\pmod{r}.$$
This equation has at most two solutions for $1\leq \alpha\leq r$. Thus, again by equation \eqref{delta},

$$\left\vert S_2(r,q,m_2;b,c,d)\right\vert=\left\vert \displaystyle\sum_{\alpha=1}^{r}e\left(\dfrac{b\alpha}{r}\right)\delta(c+dm_2\overline{q}\alpha^2,r)\right\vert\leq 2r,$$
which completes the proof of \eqref{B} in Lemma \ref{expsums}.

We proceed to prove the inequalities \eqref{C} and \eqref{D}
Analogously to the previous case, we have

$$S_2(r^f,q,m_2;b,c,d)=\displaystyle\sum_{\alpha=1}^{r^f}e\left(\dfrac{b\alpha}{r^f}\right)\delta(c+dm_2\overline{q}\alpha^2,r^f).$$
We write $(c,r^f)=r^s$, $(dm_2,r^f)=r^t$, So we have $0\leq s,t\leq f$\\

If $s<t$, for any $\alpha$, the biggest power of $r$ that divides $c+dm_2\overline{q}\alpha^2$ is always $r^s$. So the $\delta$ symbol is always zero. Hence the sum is itself always zero.\\

So, we suppose $s\geq t$.\\
In this case, we write $c=r^s\tilde{c}, dm_2=r^t\tilde{d}$. Then, the condition $r^f\mid c+dm_2\overline{q}\alpha^2$ is equivalent to

\begin{equation}\label{*}
r^{f-t}\mid r^{s-t}\tilde{c}+\tilde{d}\overline{q}\alpha^2.
\end{equation}
Notice that for any $\alpha$ for which \eqref{*} is true, we must have

$$r^u\mid \alpha,$$
where $u=\lceil\frac{s-t}{2}\rceil$. We write $\alpha=r^u\tilde{\alpha}$, with $1\leq \tilde{\alpha}\leq r^{f-u}$. Condition \eqref{*} now translates to

\begin{equation}\label{**}
r^{f-s}\mid \tilde{c}+r^{2u-s+t}\tilde{d}\overline{q}\tilde{\alpha}^2,
\end{equation}
which has at most $2$ solutions for $\tilde{\alpha}\!\!\!\pmod{r^{f-s}}$, if $r$ is odd.

\begin{rmk}\label{r=2}
If $r=2$ the quadratic equation above can have up to $4$ solutions $\pmod{r^{f-s}}$, which is in fact the only difference between the cases $r$ odd and $r=2$.
\end{rmk}
In the following, we write down the proof of \eqref{C}. The proof of \eqref{D} is exactly the same but taking into account Remark \ref{r=2}. We now have

\begin{align}\label{S2end}
S_2(r^f,q,m_2;b,c,d)&= r^f\displaystyle\sum_{\tilde{\alpha}=1}^{r^{f-u}}{}^{'}e\left(\dfrac{br^u\tilde{\alpha}}{r^f}\right)\notag\\
&=r^f\displaystyle\sum_{\tilde{\alpha}=1}^{r^{f-u}}{}^{'}e\left(\dfrac{b\tilde{\alpha}}{r^{f-u}}\right)\notag\\
&=r^f\displaystyle\sum_{\tilde{\alpha}=1}^{r^{f-s}}{}^{'}e\left(\dfrac{b\tilde{\alpha}}{r^{f-u}}\right)\displaystyle\sum_{h=0}^{r^{s-u}}e\left(\dfrac{br^{f-s}h}{r^{f-u}}\right)\notag\\
&=r^f\delta(b,r^{s-u})\displaystyle\sum_{\tilde{\alpha}=1}^{r^{f-s}}{}^{'}e\left(\dfrac{b\tilde{\alpha}}{r^{f-u}}\right),
\end{align}
where the $'$ in the sum means that we only sum over the $\tilde{\alpha}$ satisfying \eqref{**}. From \eqref{S2end}, we see that $S_2(r^f,q,m_2;b,c,d)$ is zero unless

\begin{equation}\label{power-b}
r^{s-u}\mid b.
\end{equation}
So, we assume \eqref{power-b}. Then, \eqref{S2end} gives the inequality

\begin{equation}\label{S2endend}
\vert S_2(r^f,q,m_2;b,c,d)\vert \leq 2r^{f+s-u}.
\end{equation}
And since $u\geq \frac{s-t}{2}$,

$$f+s-u\leq f+\frac{s+t}{2}\leq \frac{3f}{2}+\frac{t}{2}.$$
As a consequence, \eqref{S2endend} becomes

\begin{equation}
\vert S_2(r^f,q,m_2;b,c,d)\vert\leq 2r^{\frac{3f}{2}}(r^t)^{\frac{1}{2}}.
\end{equation}
We want to prove that

\begin{equation}\label{gcd}
(r^f,b,c,dm_2)=r^t.
\end{equation}
Since $t\leq s$ holds, equation \eqref{power-b} tell us that we only need to prove that
\begin{equation}\label{t-smallest}
t\leq s-u 
\end{equation}

If $s$ and $t$ have the same parity, $u=\frac{s-t}{2}$ and \eqref{t-smallest} is equivalent to

$$t\leq s$$
which we are assuming.

If $s$ and $t$ have different parity, $u=\frac{s-t+1}{2}$ and \eqref{t-smallest} is equivalent to

$$t+1\leq s$$
which is a consequence of $t\leq s$ and the fact that the parities are distinct.
Thus, \eqref{gcd} is true. This completes the proof of \eqref{C} and \eqref{D}

At last, for \eqref{zero}, \eqref{A}, we write

\begin{align*}
S_1(p,q,m_2;b,c,d)&=\displaystyle\sum_{\alpha,\beta=1}^{p}\displaystyle\sum_{h=1}^{p-1}\left(\dfrac{h}{p}\right)e\left(\dfrac{b\alpha+c\beta+d\overline{q}(m_2\alpha^2\beta-h)}{p}\right)\\
&=\left(\displaystyle\sum_{h=1}^{p-1}\left(\dfrac{h}{p}\right)e\left(\dfrac{-d\overline{q}h}{p}\right)\right)S_2(p,q,m_2;b,c,d),
\end{align*}
and the result follows from \eqref{B} and the well-known bound for the Gauss sum, except, possibly, when $d=p$. But in this case,

$$\displaystyle\sum_{h=1}^{p-1}\left(\dfrac{h}{p}\right)e\left(\dfrac{-dm_2\overline{q}h}{p}\right)=\displaystyle\sum_{h=1}^{p-1}\left(\dfrac{h}{p}\right)=0,$$
this gives \eqref{zero} and thus, it also gives \eqref{A} in this case.
\end{proof}

\end{document}